\documentclass[12pt]{amsart}

\usepackage{amssymb, amscd, amsmath, amsthm, epsf, epsfig, latexsym,psfrag}

\usepackage{pb-diagram}
\usepackage{mathrsfs}

\def\doub{D}
\def\Z{\mathbb{Z}}
\def\R{\mathbb{R}}
\def\Q{\mathbb{Q}}
\def\C{\mathbb{C}}

\def\calm{\mathcal{M}}

\newcommand{\Hrest}{H_1}

\newcommand{\sss}[1]{{\scriptscriptstyle{#1}}}
\newcommand{\branch}{\Sigma}
\newcommand{\inst}{$\mathscr{I}$}
\newcommand{\surg}[3]{S^3_{\scriptscriptstyle\frac{#1}{#2}}(#3)}
\newcommand{\os}{Ozsv{\'a}th-Szab{\'o}}

\newcommand{\sign}{\operatorname{sign}}
\newcommand{\lineb}{{L}_e}
\newcommand{\modE}{\mathcal{M}(W,e)}

\newcommand{\Tr}{\operatorname{Tr}}
\newcommand{\Ind}{\operatorname{Ind}}

\newcommand{\Sing}{\operatorname{Sing}}

\newcommand\cs{{\operatorname{cs}}}
\newcommand\gl{{\operatorname{gl}}}

\newtheorem{theorem}{Theorem}[section]
\newtheorem{lemma}[theorem]{Lemma}
\newtheorem{corollary}[theorem]{Corollary}
\newtheorem{prop}[theorem]{Proposition}

\theoremstyle{definition}
\newtheorem{definition}[theorem]{Definition}

\numberwithin{equation}{section}

\begin{document}

\title{Instantons, Concordance,  and Whitehead doubling }

\author{Matthew Hedden}
\author{Paul Kirk} 
 
\thanks{This work was supported in part by the National Science Foundation under Grants  0604310, 0706979, and 0906258}
 
\address{Paul Kirk: Department of Mathematics, Indiana University, Bloomington, IN 47405  }\email{pkirk@indiana.edu}
 \address{Matthew Hedden: Department of Mathematics, Michigan State University, MI 48824 }
\email{mhedden@math.msu.edu}

 \date{\today}

\begin{abstract} We use moduli spaces of instantons and Chern-Simons invariants of flat connections to prove that the Whitehead doubles of $(2,2^n-1)$ torus knots are independent in the smooth knot concordance group; that is, they freely generate a subgroup of infinite rank.  \end{abstract}

\maketitle
 
\section{Introduction}

Call two knots in the 3-sphere {\em concordant} if they arise as the boundary of a smooth and properly embedded cylinder in the 3-sphere times an interval.   Concordance is clearly an equivalence relation.  Modulo this relation, the set of knots forms an abelian group $\mathcal{C}$, with the role of addition played by connected sum and inverse given by considering the mirror image, with reversed orientation.  This concordance group is a much studied object, well motivated by its role as a gateway into the mysterious world of 4-dimensional topology. Indeed, even in this relative situation of studying 3-dimensional manifolds in relation to the 4-dimensional manifolds they bound, one can observe the distinction between the smooth and topological categories in dimension four.  Moreover, the group structure afforded by passing to concordance paves a clearer path through the often intractable field of knot theory.   

Our results focus on a particular satellite operation, (positive, untwisted) Whitehead doubling, and its effect on the concordance group, see Figure \ref{figtref} and Section \ref{subsec:prelim} for a definition.   Our motivation comes from the following conjecture.  To state it recall that a knot is {\em slice} if it is concordant to the unknot or, equivalently, if it bounds a smooth and properly embedded disk in the 4-ball.

\vskip.1in

\noindent {\bf Conjecture} \cite{Kirby}: The Whitehead double of a knot $K$  is slice if and only if $K$ is slice.
\vskip.1in

This conjecture has received considerable attention in the $30$ years since it was stated.   Part of this interest is explained by the fact that it is resoundingly false in the topological category.   To understand this, note that one can define {\em topological} concordance and sliceness  by considering  flat embeddings; that is,  topological embeddings of a cylinder times a disk.  In this context, Freedman's results on 4-dimensional surgery theory show that the untwisted Whitehead double of any knot (and, more generally, any knot with Alexander polynomial one) is topologically slice \cite{Freedman}.   In contrast, it was first observed by Akbulut \cite{Akbulutconf} that  Donaldson's theorem on the diagonalizability of definite intersection forms for closed smooth 4-manifolds implies that the Whitehead double of the trefoil is not smoothly slice.   Thus the conjecture highlights the remarkable distinction between the categories.

\begin{figure}[h]
 
\begin{center}
 \includegraphics[width=150pt]{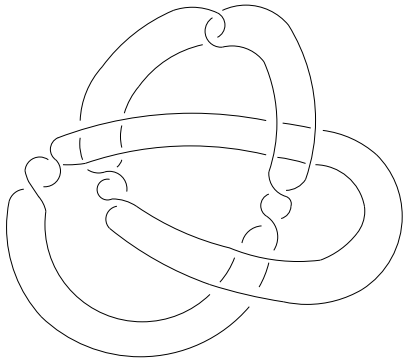}
\caption{ $\doub(T_{2,3})$, the positive untwisted Whitehead double of the right-handed trefoil, the (2,3) torus knot. \label{figtref} }
\end{center}
\end{figure}

Since the initial example, our knowledge of the conjecture in the smooth setting has considerably expanded \cite{Cochran-Gompf,Akbulut, Rudolph,Rudolph1995,LN2006, MO, Doubling}.  These results have all been aimed at extending the family of knots for which the conjecture holds.  With the exception of some recent proofs of prior results using Rasmussen's concordance invariant, all efforts have relied on the analytic techniques of gauge theory or Floer homology.  To date, the widest class of knots for which the conjecture holds are those whose Ozsv{\'a}th-Szab{\'o} concordance invariant is positive \cite{FourBall,Doubling}.  This class includes, for instance, the Whitehead double of any of its members, so the iterated doubles of any knot in this family are not not slice.  

 The purpose of this article is to begin investigation into a generalization of the conjecture.  To state it, observe that one direction of the conjecture is clear: if a knot is slice, then its Whitehead double is also slice.  Indeed, given a concordance between knots $K$ and $J$, the Whitehead doubles of $K$ and $J$ will be concordant by a cylinder which ``follows along" the original cylinder.   This shows that Whitehead doubling, or any satellite operation, induces a function
$$ D:  \mathcal{C} \rightarrow \mathcal{C},$$
from the concordance group to itself.  Although far from a homomorphism, one can ask about the structure of this function.  In particular, the conjecture is equivalent to the statement ``$D^{-1}( 0 ) = 0.$" 

 Let us call a set of  knots {\em independent} if every $N$ element subset generates a free abelian subgroup of $\mathcal{C}$ of rank $N$.    We can generalize the conjecture above. \vskip.1in

\noindent {\bf Conjecture}:    Whitehead doubling preserves independence in the concordance group.
\vskip.1in

A natural test case for the conjecture is provided by positive torus knots.  These knots can be seen to be independent by way of classical signature invariants \cite{Litherland}.  \newpage

Our primary result verifies the conjecture  for an infinite collection of torus knots.  

\vskip.1in

\noindent{\bf Theorem  1.} {\em \ The Whitehead doubles of the $(2, 2^n-1)$ torus knots, where $n=2,3,  \cdots \infty$   are  independent in the smooth concordance group.} 
\vskip.1in

In fact, our technique proves a stronger result.  Namely, we obtain independence of Whitehead doubles of the $(p_i,q_i)$ torus knots for any sequence $\{(p_i,q_i)\}_{i=1}^\infty$ of pairs of relatively prime integers   satisfying 
$$ {p_nq_n(2p_nq_n-1)}>{p_{n-1}q_{n-1}(4p_{n-1}q_{n-1}-1)}.$$
See Theorem \ref{mainresult} for the precise statement.

As the title suggests, the proof of the theorem takes place in the context of gauge theory, and makes use of a slice obstruction coming from the topology   of moduli spaces of instantons for $SO(3)$ bundles over 4-manifolds.   In the presence of a homology ball bounded by the branched covers of a collection of Whitehead doubles, we will construct a non-empty moduli space of instantons with an odd number of singular points. Each singular point has a neighborhood homeomorphic $(0,1]$, from which it follows that the moduli space is a 1-manifold with an odd number of boundary components.  If a certain numerical invariant of the branched covers,  the {\em minimal Chern-Simons invariant of   flat $SO(3)$ connections}, is large enough, it follows that the moduli space is compact, hence contradicting the existence of the putative homology ball.   The numerology of the torus knots involved in our theorem is explained by our calculations of the corresponding growth rate of the minimal Chern-Simons invariant.

 Computing Chern-Simons invariants of flat   connections on 3-manifolds requires   knowledge of the flat moduli space of the 3-manifold.  Theorem \ref{lem6.1}, which is  of independent interest, accomplishes this in the case at hand.  It states that the moduli space of flat $SU(2)$ connections (or, equivalently, the space of conjugacy classes  of $SU(2)$ representations of the fundamental group) of the 2-fold branched cover of a  Whitehead double   has a particularly simple structure. 

\medskip

\noindent{\bf Theorem \ref{lem6.1}.}  {\em Let $D(K)$ be the Whitehead double of a knot $K$, and let $\Sigma(D(K))$ denote the 2-fold branched cover of the 3-sphere, branched over $D(K)$. 
Then every $SU(2)$ representation of $\pi_1(\Sigma(D(K)))$ restricts to an abelian representation on the 2-fold branched  cover of the tubular neighborhood of $K$,   branched over the Whitehead double of its core. }

\medskip

 To put Theorem 1 in context, we should point out that infinite generation of the subgroup of $\mathcal{C}$ generated by topologically slice knots was proved by Endo using similar obstructions \cite{Endo}.   More recently, this subgroup was shown to admit an infinitely generated quotient \cite{HLR}. Almost nothing was known, however, regarding independence of doubled knots, though it follows from work of Manolescu-Owens \cite{MO} and Livingston \cite{Livingston} that the span of Whitehead doubles  contains a free abelian group of rank $2$.  This result uses Heegaard Floer theory, and we find it surprising that despite the many successes of these invariants, Theorem 1 seems out of reach to these more modern techniques. This complementary nature of the two schools of gauge theory seems quite striking, and motivates further study.  

It is also interesting to compare our results with recent work of Cochran, Harvey, and Leidy \cite{CHL}, who consider the effect of various satellite operations on the topological concordance group.   In particular, they conjecture that many such operations are injective, and use them to prove infinite generation of many of the quotients in the Cochran, Orr, Teichner filtration of the topological concordance group \cite{COT}.  Of course by Freedman's theorem, Whitehead doubling induces the zero map on the topological concordance group, so our theorem is in a different realm entirely.   In fact, Theorem 1 provides the first example of a satellite operation whose image in $\mathcal{C}$ has infinite rank, but whose image in the topological concordance group has finite rank (zero).  The techniques and tools developed here  can be viewed as a first step towards a systematic probing of the structure of the smooth concordance group through satellite operations.

\vskip 0.1in

\noindent{\bf Outline:} In the next section, we briefly review the method used to obstruct sliceness of linear combinations of Whitehead doubles.  While the obstruction comes from gauge theory, we hope that our exposition will enable those unfamiliar with this machinery to understand the general technique,  and will indicate how one may apply it for a variety of further applications.    Of particular interest is Theorem \ref{thm:indcriterion}, which provides a general criterion for a collection of homology spheres to be independent in the $\Z/2$-homology cobordism group.  

 Section \ref{2fold} is devoted to a topological analysis of the 2-fold branched covers of Whitehead doubles, with an eye towards implementing the instanton obstruction.  In particular, we analyze the JSJ decomposition of the 2-fold branched covers, and use this to construct cobordisms from them to manifolds whose interaction with gauge theory is more easily understood.

Section \ref{section:proof}, the technical heart of the article,  uses these results to prove Theorem 1.  A key step here is the computation of the minimal Chern-Simons invariants of the branched double covers.   This step is arguably the most challenging   of the article, and is informed by  techniques for analyzing the Chern-Simons invariants of flat connections on 3-manifolds decomposed along tori as developed in   \cite{KK1, KK2, auckly}.   We also prove Theorem 4.1, which is invaluable in this analysis.

 Having proved   Theorem 1, we take a step back in Section \ref{compare} to compare our technique with those afforded by more recently discovered invariants, most notably those from Heegaard Floer homology and Khovanov homology.  As mentioned, it seems very difficult to prove our theorem with these newer invariants, and we discuss reasons for this difficulty and speculate on the feasibility of an alternate proof.  

\vskip 0.1in

\noindent {\bf Acknowledgements:} It is our pleasure to thank Chuck Livingston and Danny Ruberman for their interest, and many helpful conversations.

\section{The instanton cobordism obstruction}\label{gauge}
A useful method for studying concordance is to apply a homomorphism from $\mathcal{C}$ to a 3-dimensional bordism group, and study the image of the relevant concordance classes in this latter group.   Theorem $1$ will be proved in this context, using the well-known observation that the 2-fold branched cover of a slice knot bounds a 4-manifold with the $\Z/2$-homology of a 4-ball (see Lemma \ref{branchslice} below).  Thus to obstruct sliceness it suffices to obstruct a 3-manifold, the 2-fold branched cover, from bounding a $\Z/2$-homology ball.  To this end, we outline a technique for showing that a disjoint union of oriented 3-manifolds cannot bound a negative definite 4-manifold $Q$ with $\Hrest(Q;\Z/2)=0$.   The general strategy was developed by Fintushel-Stern \cite{FSpseudo,FS}, and refined by Furuta \cite{Furuta}.  Their work took place in the context of $SO(3)$ gauge theory on pseudofree orbifolds.  Building on work of Mati{\'c} and Ruberman \cite{matic,ruberman}, we recently recast the technique entirely in terms of gauge theory on manifolds with cylindrical ends \cite{HK}. We choose to adopt this latter perspective throughout.

\subsection{The general idea from a topological perspective}
We begin by describing the argument from a topological perspective, then move on to briefly describe the gauge theory underpinning the technique.  Our discussion will take place in a somewhat simplified setting, but one which will be more than sufficient to prove Theorem $1$.  For more general obstructions  and further details, we refer the reader  to \cite{HK}.

The obstruction results from an interplay between two gauge theoretic notions, which we call {\em Property \inst} and {\em Chern-Simons bounds}, respectively.    To describe the former, let $W$ be a negative definite 4-manifold with $\Hrest(W;\Z/2)=0$.  To a class  $e\in H^2(W;\Z)$ one can associate the $SO(3)$ bundle over $W$ obtained by stabilizing the unique $SO(2)$ bundle with Euler class $e$.
Roughly, we say that  $(W,e)$ has  {\em Property \inst} if  the moduli space of instantons on this $SO(3)$ bundle is  positive dimensional and has  an odd number of singular points (see the next section for a more precise definition).      Points in this moduli space are gauge equivalence classes of solutions to the  anti self-dual Yang-Mills equations which limit, near $\partial W$,  to a fixed flat connection determined by $e|_{\partial W}$.  A feature of these equations is that their solutions come equipped with a natural notion of energy which, for the moduli spaces involved, is given by $-e\cdot e\in \Q$.  Here $e\cdot e$ denotes the cup square of a lift of $e$ to $H^2(W,\partial W;\Q)$, under the intersection form. 
 
The key fact about Property \inst  \ is that it is preserved under negative definite cobordism.  The following proposition will be proved in Subsection \ref{subsec:cobprop}.

\begin{prop} \label{prop:defcob}
Suppose $(W,e)$ has  Property \inst and $Y\subset \partial W$ is an integer homology sphere. For a negative definite 4-manifold $Q$ satisfying $\Hrest(Q;\Z/2)=0$ and $-Y\subset \partial Q$, let  $X=Q\underset{Y}\cup W$ and   $\tilde{e}=0\oplus e\in H^2(X)\cong H^2(Q)\oplus H^2(W)$.   Then $(X,\tilde{e})$ has Property \inst.
\end{prop}

Coupled with a compactness criterion for instanton moduli spaces,   Proposition \ref{prop:defcob} gives rise to a cobordism obstruction.  In the form we use, the compactness theorem is a combination of fundamental results of Uhlenbeck \cite{UhlenbeckCompactness}, and Floer, Taubes, and Morgan-Mrowka-Ruberman \cite{FloerInst,TaubesAsympPer, MMR}, respectively.  Their theorems analyze the way in which a sequence of instantons in a given moduli space  can fail to converge.   
The main idea is that any divergent sequence of instantons must take with it a non-zero amount of energy, and that this energy can be lost in two ways.  The first is through bubbling,   a phenomenon which requires $4k$ units of energy, $k\in \mathbb{N}$  \cite{UhlenbeckCompactness}.  The second arises from the non-compactness of the $4$-manifold in question (for analytic reasons, infinite cylindrical ends are added to boundary components of $4$-manifolds), and is often referred to as  ``energy escaping down an end.''      Thus  a divergent sequence is prohibited if the energy of the instantons in a given moduli space is small enough to prohibit either of these phenomena.   

For the latter issue, the quanta of energy which can escape down an end are governed by Chern-Simons invariants.  To make this precise, recall that a  flat connection  $\gamma$ on an $SO(3)$ bundle  over a 3-manifold $Y$ has a  Chern-Simons invariant $\cs(Y,\gamma)\in \R/4\Z\cong (0,4]$ 
(see the next section for a definition).  Given a closed,   oriented 3-manifold $Y$ and class $e\in H^2(Y;\Z)$ we will consider, as above, the $SO(3)$ vector bundle obtained by stabilizing the $SO(2)$ bundle with Euler class $e$.  This bundle, which we denote $E_e$, admits a unique (up to gauge equivalence) reducible flat connection $\alpha_e$ compatible with the given reduction of structure.

 Denote  by $\tau(Y,e)$ the minimum of the differences$$\cs(Y,\gamma)- \cs(Y,\alpha_e) \in (0,4],$$ where $\gamma$ ranges over all flat connections on $E_e$ (again, see the next section for a more precise definition of $\tau(Y,e)$). We have the following compactness criterion for our moduli spaces, arising from the fact that the amount of energy that can  escape down an end in a divergent sequence must equal $\cs(Y,\gamma)- \cs(Y,\alpha_e)$ for some $\gamma$.

\begin{prop}  {\em(\cite[Proposition 2.9]{HK} cf. \cite{UhlenbeckCompactness,FloerInst,TaubesAsympPer,MMR})} \label{prop:compactness} Let $W$ be a negative definite 4-manifold and $e\in H^2(W)$. Let $\modE$ be the moduli space of instantons determined by $e$. If each component $Y\subset \partial W$   satisfies
\begin{equation}\label{eq:csbound} 0< -e\cdot e< \tau(Y,e|_Y)\le 4,\end{equation}  then $\modE$ is compact. \qed
\end{prop}

Propositions \ref{prop:defcob} and \ref{prop:compactness} lead immediately to a cobordism obstruction.  Suppose we wish to show that an integer homology 3-sphere $Y$ doesn't bound any 4-manifold with the $\Z/2$-homology of a 4-ball.  If we know that $Y$ {\em does} bound a pair $(W,e)$ with Property \inst\ then, according to  Proposition \ref{prop:defcob}, $(X=Q\cup_{Y} W,\tilde{e})$ will satisfy Property \inst \ for any putative $\Z/2$-homology ball $Q$ with $-\partial Q=Y$.  This produces a positive dimensional moduli space $\mathcal{M}(X,\tilde{e})$ with an odd number of singular points.  If the Chern-Simons bound $$0< -\tilde{e}\cdot \tilde{e}=-e\cdot e< \tau(M,e|_M)\le 4,$$ is satisfied for each component of $\partial W$ other than $Y$, it follows from Proposition \ref{prop:compactness} that  $\mathcal{M}(X,\tilde{e})$ is compact.  Now the singular points in $\mathcal{M}(X,\tilde{e})$ have neighborhoods homeomorphic to cones on $\mathbb{C}P^N$, where $2N+1=\mathrm{dim}\ \mathcal{M}(X,\tilde{e})$.  When $N=2k$  we immediately have a contradiction, since an odd number of $\mathbb{C}P^{2k}$'s cannot bound a $(2k+1)$-manifold.  A refined argument handles the case when $N$ is odd, but for the purposes of Theorem 1, we need only the case $N=0$; that is,  the moduli spaces used in the proof of our main theorem  are 1-manifolds with boundary. 

 In the subsequent sections, we turn to a more detailed description of Property \inst \ and the Chern-Simons invariants.  In particular, we will see that the boundary of a 4-manifold satisfying Property \inst  \ is naturally divided into 2 types: those components along which we can glue negative definite cobordisms while retaining Property \inst \ (Proposition \ref{prop:defcob} indicates that homology spheres fall into this category), and those components $M$ satisfying the Chern-Simons bound,  \eqref{eq:csbound}. This perspective will be quite useful, and leads to a criterion (see  Theorem \ref{thm:indcriterion}) for a collection of homology spheres to be independent in the homology cobordism group.  The challenge of the present article is to verify that branched double covers of Whitehead doubles of certain torus knots satisfy this criterion.

Before delving any further into gauge theory, we remark that while constructing moduli spaces of instantons with special properties may seem intimidating to the uninitiated, Property \inst \  can often be easily verified using purely topological techniques.  Indeed, in addition to its role in the cobordism obstruction, Proposition \ref{prop:defcob} can also be used to show that many 3-manifolds bound 4-manifolds satisfying Property \inst, even if their interaction with gauge theory is initially obfuscated; all we must do is find an appropriate cobordism from the 3-manifold in question to one known to bound a 4-manifold with Property \inst.   In fact, for the purpose of proving Theorem $1$, it will suffice to know that certain Seifert manifolds bound 4-manifolds with Property \inst, with appropriate Chern-Simons bounds.

\begin{prop} {\em(}\cite[Section 3]{HK}, c.f. \cite{FSpseudo}{\em )} Let $p,q,k>0$, and $gcd(p,q)=1$. \label{truncated} There is a pair $(X,e)$ with Property \inst, with  $e\in H^2(X)$ satisfying $$-e\cdot e=\tfrac{1}{pq(pqk-1)}.$$
 The boundary of $X$ consists of the Brieskorn homology sphere $\Sigma(p,q,pqk-1)$ and three lens spaces $Y_j,  j=1,2,3$, which  satisfy the Chern-Simons bounds   
$$ \tau(Y_1, e_1)\ge\tfrac{1}{p}, \ \ \  \tau(Y_2, e_2)\ge\tfrac{1}{q}, \ \ \  \text{ and } \ \tau(Y_3, e_3)\ge\tfrac{1}{pqk-1}$$
for any class $e_i\in H^2(Y_i)$. \qed
\end{prop}

\subsection{A bit of gauge theory}\label{subsec:gaugetheory}

In this subsection we fill in some of the gauge theoretic gaps left by the discussion of the previous subsection.  Our purpose here is to give a more precise definition of what it means for a 4-manifold to have Property \inst, and to describe Chern-Simons invariants and some of their key features.  The treatment will be intentionally terse, and a more detailed exposition is the primary purpose of the companion article \cite{HK}.  The discussion here could be viewed as a guide and summary of \cite{HK}. 
\medskip

As above, and for the remainder of this section, let $W$ be a negative definite 4-manifold satisfying $\Hrest(W;\Z/2)=0$.   We adopt the convention that a 4-manifold is negative definite if the cup-square of any non-torsion cohomology class is a negative multiple of the fundamental class.   With this convention, a 4-manifold with $H_2(W;\Q)=0$ is negative definite.  Also observe that since $\Hrest(W;\Z/2)=0$, being negative definite implies $\partial W$ is a union of rational homology spheres.  Thus all 3-manifolds in our discussion satisfy $H^1(Y;\Z)=0$.

 Given a 4-manifold $W$ we can associate a relative Pontryagin number to any pair $(E,\alpha)$, where $E$ is an $SO(3)$ bundle over $W$ and $\alpha$ is a flat connection on $E|_{\partial W}$.  To do this, consider a connection $A$ on $E$ which restricts to  to $\alpha$ on the boundary, and define the {\em relative Pontryagin number} $p_1(E,\alpha)$ by $$p_1(E,
 \alpha)=-\frac{1}{8\pi^2}\int_W \Tr(F(A)\wedge F(A)),$$ 
 where $F(A)$ denotes the curvature 2-form of $A$.  As the notation suggests, this number is independent of the choice of connection extending $\alpha$. 
 
A  class  $e\in H^2(W;\Z)$  determines  the unique $SO(2)$ vector bundle   $\lineb$ over $W$ with Euler class $e$.  Stabilizing $\lineb$ by taking a sum with the trivial 1-dimensional real line bundle $\epsilon$ produces an $SO(3)$ vector bundle, denoted $E_e\cong \lineb\oplus \epsilon$.   

 Now $\lineb|_{\partial W}$ possesses a flat connection $\beta_e$ which is unique up to gauge transformation  (Lemma $2.10$ of \cite{HK}).  We can stabilize $\beta_e$ by the trivial connection on  $\epsilon$ to yield a flat connection on $E_e|_{\partial W}$.  We will denote this stabilized  flat connection by $\alpha_e$.

For the pair $(E_e,\alpha_e)$ we consider here, $p_1(E_e,\alpha_e)$ depends only on  $(W,e)$.  It  can be calculated in terms of the intersection pairing of $W$:   $$p_1(E_e,\alpha_e)=-e\cdot e\in \Q.$$ (See Section 2.6 \cite{HK} for details on the extended rational-valued intersection pairing on $H^2(W;\Z)$, and Proposition 2.13 of \cite{HK} for the identification $p_1(E_e,\alpha_e)=-e\cdot e$.)

 \medskip

 Closely related to $p_1(E,\alpha)$ is the relative Chern-Simons invariant.   Given a rational homology 3-sphere $Y$ and class $e\in H^2(Y;\Z)$, consider the 4-manifold $[0,1]\times Y$ and the corresponding class (which we continue to denote by) $e\in H^2([0,1]\times Y;\Z)$. Let $ E_e$ denote  the  $SO(3)$ bundle over $[0,1]\times Y$ determined by $e$, as described above.  For any flat $SO(3)$ connection $\gamma$ on $E_e|_{\{0\}\times Y}$, choose a connection $A$ on $E_e$ whose restriction to $\{t\}\times Y$ equals $\gamma$ for $t$ near $0$ and equals $\alpha_e$ for $t $ near $1$. Define 
\begin{equation}\label{cs}
\cs(Y,\gamma, \alpha_e)=-\frac{1}{8\pi^2}\int_{[0,1]\times Y}  \Tr(F(A)\wedge F(A)).
\end{equation}
The set  of values taken modulo 4, $\{\cs(Y,   \gamma,\alpha_e)\ | \ \gamma\text{ a flat connection on } E_e \}\subset \R/4\Z$,
is finite, since $\cs(Y,-,\alpha_e)$ is a  a locally constant function on the space of gauge equivalence classes of flat connections on $E_e$, a space with finitely many path components. Identifying $\R/4\Z$ with $(0,4]$ in the obvious way,    define the {\em minimium relative Chern-Simons invariant} $ \tau(Y,e)\in(0,4]$ by
$$ \tau(Y,e) =\min \{\cs(Y,   \gamma,\alpha_e)\ | \ \gamma\text{ a flat connection on } E_e \}.$$  When $e=0$ we  write $\tau(Y)=\tau(Y,e)$.

The Chern-Simons invariant has a useful cobordism property.  To describe it, suppose we have a flat connection $\gamma$ on an $SO(3)$ bundle $E$ over a 3-manifold $Y$.   Consider any extension of  $E$ to a bundle $\tilde{E}$ over some 4-manifold $W$ with $\partial W=Y$.   Then $p_1(\tilde{E},\gamma)$ modulo $\Z$ depends only on the pair $(Y,\gamma)$, and not on our choice of $W$ or the extension of $E$.   This is a consequence of the fact that  the first Pontryagin number of a bundle over a closed 4-manifold is an integer.  In light of this, we denote  $p_1(\tilde{E},\gamma)$ modulo $1$,  by $\cs(Y,\gamma)$.  This becomes useful, due to the obvious equality 
 $$\cs(Y,\gamma,\alpha_e)= \cs(Y,\gamma)-\cs(Y,\alpha_e) \text{ mod }\Z.$$
With care, one can retain mod $4$ information, but for our purposes it is sufficient to estimate 
 $\cs(Y,\gamma,\alpha_e)$ and $\tau(Y,e)$ modulo $\Z$.  In fact, many of our estimates in Section  
\ref{section:proof}  will come from the following lemma, which is a simple consequence of the definitions.

\begin{lemma} \label{observation} Let $Y$ be a rational homology 3-sphere.    Suppose  that for every flat connection $\gamma$ on every $SO(3)$ bundle over $Y$,  $cs(Y,\gamma)$ is a rational number whose denominator divides $p$.  Then $\tau(Y,e)\ge \frac{1}{p}$ for all $e\in H^2(Y)$.

\end{lemma}
 
\noindent For example, if $Y$ has finite fundamental group of order $p$, then $\tau(Y,e)\ge \frac{1}{p}$ for all $e\in H^2(Y)$.

 \medskip
 
Having dispatched the Chern-Simons invariants, we turn our attention to  Property \inst.  We would like this to mean that a certain moduli space of instantons $\modE$ is positive dimensional and has an odd number of singularities. In order to understand this, we must first understand the definition of $\modE$.

For analytic reasons, one attaches ends to $W$ which are isometric to $[0,\infty)\times \partial W$, with corresponding extensions of the bundle $E_e|_{\partial W}$.  We then have a flat connection $\alpha_e$ on the ends $[0,\infty)\times \partial W$.
The moduli space $\modE$ consists of gauge equivalence classes of {\em instantons} on the bundle  $E_e$ which converge exponentially (with respect to appropriate Sobolev norms) along the ends to $\alpha_e$.  Instantons, by definition, are connections $A$ on $E_e$ whose curvature satisfies the elliptic partial differential equation $$F(A)=-\ast F(A). $$ See  Sections $2.1$ and $2.3$ of \cite{HK} for further analytic details and references.  The exponential decay condition ensures that the integral
 $$ -\frac{1}{8\pi^2} \int_{W\cup  ([0,\infty)\times \partial W)} \Tr(F(A)\wedge F(A)) $$
 converges and equals $p_1(E_e,\alpha_e)$.

The natural notion of energy for a  connection is given by the (normalized) $L^2$ norm of its curvature $$   ||F(A)||^2_{L^2}= \frac{1}{8\pi^2}  \int_{W\cup  ([0,\infty)\times \partial W)} \Tr(F(A)\wedge \ast F(A)),$$ which, for an instanton, satisfies 
$$||F(A)||^2_{L^2}=p_1(E_e,\alpha_e)=-e\cdot e.$$  In particular, we see that the energy of an instanton in $\modE$ depends only on the bundle $E_e$ which supports it. 

 For the moduli space to have the requisite properties, it is necessary to assume that the restriction of the flat connection  $\alpha_e$ to each component $Y_i\subset \partial W$ is {\em non-degenerate}.  By definition, this means that the associated first cohomology group $H^1(Y_i;\R^3_{\alpha_e})$ vanishes. This holds, for instance, when  $Y_i$ is a rational homology sphere and $e|_{Y_i}$ vanishes. It also holds for any $e$ if  $Y_i$ is a lens space.  If non-degeneracy holds for each $Y_i\subset \partial W$, we will simply say that $\alpha_e$ is non-degenerate.

\medskip

Calculating the dimension of $\modE$ amounts to performing an index calculation for the linearization of the instanton equation.  In the present context, we denote this index by Ind$^+(W,e)$.  A formula for Ind$^+(W,e)$ is given by Proposition 2.6 of \cite{HK}:

\begin{equation}\label{index} \mathrm{Ind}^+(W,e)=-2e\cdot e- 3 + \tfrac{1}{2} \underset{ \{e|_{Y_i} \mathrm{\ is\  nontrivial \ } \}}\sum (3- h_{\alpha_e}(Y_i)-\rho(Y_i, \alpha_e)),\end{equation}
where  $ Y_1\sqcup\cdots\sqcup Y_n = \partial W$.  The quantities in the  summation come from the boundary terms in the Atiyah-Patodi-Singer theorem, and will be ignored for our purposes.  The key point is that they are only affected by those components of $\partial W$ on which the flat connection $\alpha_e$ is non-trivial, since $h_{\alpha_e}(Y)=3$ and $\rho(Y,\alpha_e)=0$ if $e|_{Y}=0$.

We will consider only the case that $-e\cdot e>0$.  Then, if $\Ind^+(W,e)>0$, one can perturb the Riemannian metric on $W\cup \big( [0,\infty)\times \partial W\big)$ to ensure that $\modE$ is a smooth  manifold away from a finite number of singular points whose neighborhoods are homeomorphic to a cone on $\C P^N$, where $2N+1=\Ind^+(W,e)$.  In the situations which arise in this article, $N=0$. Hence $\modE$ is a 1-manifold with boundary.

\medskip

We are now ready to make the following definition.

\begin{definition} A   compact, connected  negative definite 4-manifold $W$  with $\Hrest(W;\Z/2)=0$  has {\em Property \inst \   with respect to }  $e\in H^2(W;\Z)$  provide that
 \begin{itemize}

\item The class $e$ satisfies $-e\cdot e<2$.
\item The reducible flat connection $\alpha_e$ on $E_e|_{\partial W}$ is non-degenerate.

\item $\Ind^+(W,e)>0$.

 \item  The number of singular points of $\modE$ is odd. \end{itemize}
For brevity, we will often say {\em $(W,e)$ satisfies Property \inst}.
\end{definition}

The number of singular points of $\modE$ is essentially determined by the intersection form of $W$.  
To see how, let us use the notation $\Sing(W,e)$ for the set of singular points of $\modE$. Now define
the set 
$$C(W, e)= \left\{\begin{array}{l}
\\ e'\in H^2(W;\Z) \\ 
\\ \end{array} \right\| 
\left.\begin{array}{l}  e'\cdot e'=e\cdot e 
\\  e'=e\ \text{mod } 2 \\ e'|_{Y}=\pm e|_{Y}\text{ for each component } Y\subset \partial W \\ \end{array}\right\}/{\pm1}.$$

\noindent Note that $C(W, e)$ is finite since $W$ is negative definite.  Proposition 2.15 of \cite{HK} shows that  there is an injective function  $$\Sing(W,e)\hookrightarrow C(W,e).$$   In light of this, we abuse notation and consider $\Sing(W,e)$ as a subset of $C(W,e)$.

In certain circumstances we can easily conclude that $\Sing(W,e)=C(W,e)$. One such situation occurs if $C(W, e)$ contains a unique element (necessarily $e$), and then $\Sing(W,e)$ has a unique element as well.  Another case when a bijection is  guaranteed occurs  when $H^1(\partial W;\Z/2)=0$. In general, to each $e'\in C(W,e)$ one can associate an obstruction living in $H^1(\partial W;\Z/2)$, whose vanishing ensures that $e\in \Sing(W,e)$  (see Theorem 2.16 of \cite{HK}).  In the next section we prove Proposition \ref{prop:defcob}.  The proof shows that if $e+2a\in C(W,e)$  for a class $a\in H^2(W)$ that lifts to  a torsion class $\tilde{a}\in H^2(W,\partial W)$, then $e+2a\in \Sing(W,e)$.

\vskip0.01in
In general, the question of whether a class $e'\in C(W,e)$ lies in $\Sing(W,e)$ is purely topological. Indeed, $e'\in C(W,e)$ lies in $\Sing(W,e)$ if and only if
one can find an $SO(2)$ subbundle $$L_{e'}\subset E_e=L_e\oplus \epsilon$$ such that $L_{e'}$ has Euler class $e'$ and so that the restriction of $L_{e'}$ to each boundary component of $W$ coincides with $L_e$ as unoriented subbundles. See Lemma 2.12 of \cite{HK} for details.

\medskip

\subsection{Cobordism properties}\label{subsec:cobprop} Having introduced the necessary gauge theory, we can now be more precise regarding the behavior of Property \inst \ under cobordism, and its role as a cobordism obstruction.  We observe a natural way to partition the boundary components of 4-manifolds, which we call a cs-partition.   This notion leads to Corollary \ref{cor:defcob2},  a    refined version  of Proposition \ref{prop:defcob}.  It  also leads to a general independence criterion, Theorem \ref{thm:indcriterion}, which will be the foundational tool for this article.

\medskip

We begin by proving Proposition \ref{prop:defcob}. Recall that this result says Property \inst \  is preserved when a 4-manifold $W$ having it is enlarged by a negative definite 4-manifold $Q$, attached to $\partial W$ along a homology sphere.  This yields a manifold $X=Q\underset{Y}\cup W$.

\medskip

\noindent{\it Proof of  Proposition \ref{prop:defcob}.}   
Assume that $(W,e)$ satisfies Property \inst.  We wish to show that $(X,\tilde{e})$ satisfies Property \inst, with  $\tilde{e}=0\oplus e\in H^2(X)\cong H^2(Q)\oplus H^2(W)$. This involves showing four things, and we begin with the most challenging; namely, that the number of singular points of $\mathcal{M}(X,\tilde{e})$ is odd.  

 Denote the  boundary components of $W$ by $Y, R_1,R_2,\cdots, R_n$, and the boundary components of $Q$ by $-Y, P_1, \cdots ,P_m$. Let $T=$Torsion$(H^2(Q,\partial Q))$. Our assumption that $H_1(Q;\Z/2)=0$ implies $H^2(Q)$ has no $2$-torsion and, since $H^2(Q,\partial Q)$ injects into $H^2(Q)$, that $T$ has odd order.

 Our first step is to show  that $C(X,\tilde e)= T\times C(W,e)$.   We can decompose any $\tilde e'\in C(X,\tilde e)$ as   $\tilde e'=   2a + e+ 2b$, where $a\in H^2(Q)$ and $b\in H^2(W)$.  Note that the splitting of $H^2(X)$ which allows this decomposition is orthogonal with respect to the intersection form, and arises because $Y$ is a homology sphere. Thus 
 $$e\cdot e= \tilde{e}\cdot \tilde{e} = \tilde{e}'\cdot \tilde{e}'=4a\cdot a + (e+ 2b)\cdot (e+2b).$$
 The restriction of $2a$ to $\partial Q$ equals zero, since its restriction to   $P_i$ equals $\pm \tilde{e}|_{P_i}=0$, and any class restricts to zero on $Y$ (as it is a homology sphere).  If follows that  $2a $ lifts to $H^2(Q,\partial Q)$, and hence $a\cdot a\in \tfrac{1}{2}\Z$.  Since $Q$ is negative definite, we see that $4a\cdot a$ is a non-positive even integer.

On the other hand, since $W$ is negative definite,  $ (e+  2b)\cdot (e+ 2b) \le 0$. Hence
$$ -2<e\cdot e\leq 4a\cdot a.$$
This implies $a\cdot a=0$, so that $a$ is a torsion class.  Thus $2a$ is a torsion class lifting to $H^2(Q,\partial Q)$. Injectivity of $H^2(Q,\partial Q)\to H^2(Q)$ implies that its lift, which we also denote $2a$, is torsion; that is, $2a\in T$.   Moreover, every class in $T$ can be uniquely written in the form $2a$ (since multiplication by $2$ is an automorphism of any odd order group).
Since $a\cdot a=0$, we have $(e+2b)\cdot(e+2b)=e\cdot e$, implying that $e+2b\in C(W,e)$.  Hence $C(X,\tilde e)= T\times C(W,e)$, as asserted.  

 Next we show that $\Sing(X,\tilde{e})=T\times \Sing(W,e)$.
Suppose  first that $ e' =(e+2b)\in  \Sing(W,e)$ and choose a class $2a\in T\subset H^2(Q,\partial Q)$. We show that $2a+(e+2b)\in \Sing(X,\tilde e)$. 

The bundle $E_e=L_e\oplus \epsilon$ over $W$ admits an $\R^2$  subbundle $L_{e+2b}$ which coincides with $L_e$ (perhaps as unoriented bundles)   over the boundary components $Y, R_1,\cdots , R_n$ of $W$.  
The bundle $L_e$ (and hence also $L_{e+2b}$) is trivial over $Y$, since $Y$ is an integer homology sphere. Any trivialization is unique up to homotopy, since $[Y, SO(2)]=H^1(Y)=0$. Choose a pair of non-vanishing linearly independent sections $s_1, s_2$ of $L_e$ over $Y$ and let $s_3$ be the section of the trivial line bundle in $E_e=L_e\oplus \epsilon$. Let $\tilde{E}_{e}$ be the $\R^3$ bundle over $X$ obtained by gluing the trivial $\R^3$ bundle over $Q$  to $E_{e}$ using the trivialization provided by $s_1,s_2,s_3$ on $Y$.  Then  $\tilde{E}_{e}$ contains the subbundle $\tilde L_{e+2b}$ which agrees with $L_{e+2b}$ over $W$ and is trivial,  spanned by $s_1,s_2$, over $Q$. 

Choose a closed embedded surface $F$ in the interior of $Q$ representing $a$.   Since $a$ is a torsion class, $F$ has a neighborhood diffeomorphic to $F\times D^2$.  One can alter the section $s_3$ over this neighborhood using a map $D^2\to S^2$ which takes the boundary circle to $(0,0,1)$ and wraps the disk once around $S^2\subset \R^3$. Call this new (nonvanishing) section $s_{2a}$.  Then the perpendicular planes to $s_{2a}$ determine an $\R^2$ 
subbundle of $\tilde{E}_{e}$ which agrees with $\tilde L_{e+2b}$ outside the neighborhood of $F$ and has Euler class $2a+e+2b$. We denote it by $L_{2a+(e+2b)}\subset \tilde{E}_e\cong E_{\tilde e}$. By construction, $ L_{2a+(e+2b)}$ and  $\tilde L_{e}$ coincide  outside the neighborhood of $F$ and, in particular, on each boundary component of $X$. Lemma 2.12 of \cite{HK} then shows that $2a+(e+2b)\in \Sing(X,\tilde e)$.

Conversely, suppose that $\tilde e'=2a + (e+2b)\in \Sing(X,\tilde e)\subset T\times C(W,e)$. Then  there is an $\R^2$ subbundle ${L}_{\tilde e'}\subset {E}_{\tilde e}= L_{\tilde e}\oplus \epsilon$ which coincides with $ L_{\tilde e}$ over the boundary components of $X$.    We wish to isotope   ${L}_{\tilde e'}$ slightly so that it coincides with ${L}_{\tilde e}$ in a neighborhood of $Y$.

Since $Y$ is an integer homology sphere,  the restrictions of ${L}_{\tilde e}$ and ${L}_{\tilde e'}$ to $Y$ are trivial bundles. Choose a pair of non-vanishing linearly independent sections $s_1, s_2$ of $L_{\tilde e}$ over $Y$ and let $s_3$ be the section of the trivial line bundle in $E_{\tilde e}= L_{\tilde e}\oplus \epsilon$.  Similarly choose a pair of non-vanishing linearly independent sections $t_1, t_2$ of $L_{\tilde e'}$ over $Y$ and let $t_3$ be the section of the   trivial line bundle perpendicular to  $L_{\tilde e'}$.

Comparing the trivializations gives a well-defined map $f:Y\to O(3)$.  By reversing the sign of $t_3$, if necessary, we can assume $f$ maps to $SO(3)$.  If $f$ is nullhomotopic, then we can use a nullhomotopy to construct a gauge transformation on a collar neighborhood of $Y$ in $X$ which is the identity near the boundary.   Extending this to the rest of $X$ by the identity we arrive at a gauge transformation $g: E_{\tilde{e}}\rightarrow E_{\tilde{e}}$
 which, over $Y$, satisfies the desired equality $g(L_{\tilde e'})=   L_{\tilde e}$.

 A standard argument    using the bundle over the double of $W$ obtained by clutching with $f$ along $Y$ and the identity along the other boundary components, shows that  $4\deg f= e\cdot e -(e+2b)\cdot (e+2b)=0$.  Since $Y$ is an integer homology sphere this implies that $f$ is nullhomotopic. 
 
 Thus we may assume that ${L}_{\tilde e'}$ coincides with ${L}_{\tilde e}$ in a neighborhood of $Y$.  Restricting now to $W$ and applying Lemma 2.12 of \cite{HK} shows that $e+2b=\tilde{e}'|_W\in \Sing(W,e)$.
 
 We conclude that if the number of singular points of $\modE$ is odd, then the number of singular points of $\calm(X,\tilde{e})$ is also odd, since $T$ is an odd torsion group.
 
 The remaining aspects of Property \inst \
  are much simpler to verify. The energy inequality follows from  $-\tilde{e}\cdot\tilde{e}=-e\cdot e$ and the assumption that $(W,e)$ satisfies Property \inst.  For non-degeneracy, observe that each boundary component  of $Q$ is a rational homology sphere and $\tilde{e}|_Q=0$. Thus the reducible flat connection $\alpha_{\tilde{e}}$ is trivial on each $P_i\subset \partial Q$, hence non-degenerate.  Non-degeneracy of   $\alpha_{\tilde{e}}$ on $R_i$ follows from the fact that $\alpha_{\tilde{e}}|_{R_i}=\alpha_e|_{R_i}$, and our assumption.
 
 It remains to show  $\Ind^+(X,\tilde{e})>0$.  But this follows from the fact that  $\Ind^+(W,e)>0$, and the observation that $\Ind^+(X,\tilde{e})=\Ind^+(W,e)$.  This latter observation, in turn, is an immediate consequence of Equation \ref{index}, noting that $\tilde{e}\cdot\tilde{e}=e\cdot e$, and  that $\tilde{e}$ restricts to zero on each $P_i$.   Hence $(X,\tilde{e})$ has Property \inst, as desired.
\qed

\bigskip

Proposition \ref{prop:compactness} shows that if $(W,e)$ has Property \inst \ and $\Ind^+(W,e)=1$, then some boundary component $Y$ of $W$ must satisfy $\tau(Y,e|_Y)\leq -e\cdot e$, since otherwise $\modE$ would be a compact 1-manifold with an odd number of boundary points. In fact, if $(W,e)$ has Property \inst \ then there exists $Y\subset \partial W$ satisfying $\tau(Y,e|_Y)\leq -e\cdot e$, regardless of the index (see Theorem 2.17 of \cite{HK}).

In combination with  Proposition \ref{prop:defcob}, this suggests a strategy for obstructing the existence of certain 4-manifolds. Roughly speaking,  given a pair $(W,e)$ with Property \inst, partition the boundary components of $W$ into two sets. First,  those components $Y$ satisfying the Chern-Simons bound  $\tau(Y, e|_{Y})>-e\cdot e$ and second, the remaining components, all of which are integer homology spheres.  Now we attach negative definite 4-manifolds  to these latter components which, according to  Proposition \ref{prop:defcob}, preserves Property \inst.  Continuing in this fashion one builds a negative definite 4-manifold with a partition of its boundary into homology spheres and rational homology spheres satisfying the Chern-Simons bound.  If the remaining homology spheres were to bound a rational homology punctured ball $Q$, then gluing the two manifolds together would yield a 4-manifold satisfying Property \inst \ and Chern-Simons bounds on all its boundary components.   This results in a  compact moduli space with an odd number of singular points, contradicting the existence of $Q$.

We introduce some notation to formalize this outline.

\begin{definition}\label{piwcs} A triple $(W,e,b)$,  where $W$ is a compact 4-manifold with boundary, $e\in H^2(W)$,  and $0<b\in\R$ a positive real number, is said to admit a {\em $\cs$-partition} if the boundary of $W$ can be partitioned  into two disjoint sets,  $$\partial W= \partial _{\gl}W\sqcup \partial _{\cs}W$$
in such a manner that
\begin{itemize}
\item each component of $\partial _{\gl}W$ is an integer homology sphere, and
\item each component $Y\subset \partial _{\cs}W$ satisfies the Chern-Simons bound $\tau(Y, e|_Y)>b$.
\end{itemize}
\end{definition} 
In general $(W,e, b)$ need not admit any cs-partitions. Moreover, cs-partitions need not be unique since some homology sphere components of $\partial W$ could be considered either in $\partial_\gl W$ or $\partial _\cs W$ if their Chern-Simons invariants satisfy the appropriate bound.   Proposition \ref{prop:compactness} implies the following. 

\begin{prop}\label{prop:compactness2} Let $(W,e,b)$ be a triple admitting a $\cs$-partition with $W$ negative definite, $\partial_\gl W$ empty,  $0<-e\cdot e<4$, and $\Ind^+(W,e)$ positive. Then $\modE$ is compact and the cardinality of $\Sing(W,e)$  is even.
\end{prop}

We   use the notation $e^2=e\cdot e\in \Q$ for a class $e\in H^2(W)$; note that since $W$ is typically assumed to be negative definite, $e^2\leq 0$.
Proposition \ref{prop:defcob} has the following corollary.

\begin{corollary} \label{cor:defcob2} Suppose $(W,e)$ satisfies Property \inst \ and $(W,e,-e^2)$ admits a cs-partition.    Suppose  that $Q$ is a compact, connected, negative definite 4-manifold with $\Hrest(Q;\Z/2)=0$ and suppose $(Q,0, -e^2)$ admits a cs-partition  $\partial_\gl Q\sqcup \partial_\cs Q$.  Finally, suppose  $Y$ is an integer homology sphere such that $Y\subset \partial_\gl W$ and $-Y\subset \partial_\gl Q$.
\vskip0.03in
Then $(X=Q\underset{Y}\cup W,\tilde e)$ satisfies Property \inst \  and $(X,\tilde e, -e^2)$ admits a cs-partition with  \begin{enumerate}\label{infinite}
\item $\tilde{e}=0\oplus e \in H^2(X),$
\item $\partial_\gl X=  (\partial_{\gl} Q \sqcup \partial_\gl W)\setminus Y,$ and
\item $\partial_\cs X = \partial_\cs Q \sqcup \partial_\cs W$.
\end{enumerate}\qed
\end{corollary}

 As an illustration of these ideas,  we have the following result.
\begin{prop} \label{infinite2} Let $Y$ be an integer homology sphere.  Suppose there is a pair $(W,e)$   with Property \inst, and    $(W,e,-e^2)$ admits a   $\cs$-partition  with $Y=\partial_\gl W$. Further suppose that there is a negative definite 4-manifold $W'$, and     $(W',0,-e^2)$   admits   a   $\cs$-partition with $Y=\partial_\gl W$.   Then no multiple of $Y$ bounds a 4-manifold with the $\Z/2$-homology of a punctured ball.
\end{prop}
\begin{proof} Suppose, to the contrary, that there exists a 4-manifold $Q$ satisfying $H^i(Q;\Z/2)=0$ for $i=1,2$ and with  $\partial Q=nY$, where $n Y= \sign(n) Y\sqcup \cdots \sqcup \sign(n) Y$ denotes the disjoint union of $|n|$ copies of $Y$ or $-Y$, with the orientation of $Y$ depending on the sign of $n$.   Note that by reversing the orientation of $Q$, if necessary, we may assume that $n<0$.

Now to one  boundary component of $Q$ attach a copy of $W$ and to the remaining components attach copies of $W'$.  This  produces a manifold $X$ satisfying $H^2(X)\cong H^2(Q)\oplus H^2(W)\oplus_{i=1}^{n-1} H^2(W')$.   Let $\tilde e=0\oplus e\oplus 0\oplus 0\oplus \cdots\oplus  0$ with respect to this decomposition.  Induction and Corollary \ref{cor:defcob2} imply that $(X,\tilde e)$ satisfies Property \inst\ and that $(X, \tilde e, -e^2)$ admits a cs-partition with  with $\partial_\gl X$ empty. Proposition \ref{prop:compactness2} therefore implies that $\calm(X,\tilde e)$ is compact with an even number of singular points. But since $(X,\tilde e)$ satisfies Property \inst, 
$\calm(X,\tilde e)$ has an odd number of singular points, a contradiction.\end{proof}

The above proposition is useful in showing that a particular    knot has infinite order in the   concordance group.  There are many techniques for proving such a result (see Section \ref{compare}), but 
our broader aim is to develop a practical method to prove that a set of   knots are independent.  The following result will be instrumental in accomplishing this goal.

 \begin{theorem}\label{thm:indcriterion}
Let $\{Y_i\}_{i=1}^M$ be a set of oriented integer homology spheres. For each $i$, suppose there is  a pair $(W_i,e_i)$   with Property \inst  and a negative definite 4-manifold $W_i'$, so  that    $(W_i,e_i,-e_i^2)$   and $(W_i',0,-e_i^2)$  admit   $\cs$-partitions with $Y_i=\partial_\gl W_i=\partial_\gl W'_i$.    If, after reindexing if necessary, 
 $$  -e_i^2 < \tau(\pm Y_j)  \ \ \mathrm{for \ all\ \ }  j<i,$$
then no combination $$ n_1 Y_1 \sqcup n_2 Y_2 \sqcup ... \sqcup n_M Y_M, \ n_i\in \Z$$ bounds a 4-manifold with the $\Z/2$-homology of a punctured 4-ball.
\end{theorem}
\begin{proof}  Suppose, to the contrary, that there exists a  4-manifold $Q$ with  the $\Z/2$-homology of a punctured 4-ball and  $ \partial Q= n_1 Y_1 \sqcup n_2 Y_2 \sqcup ... \sqcup n_M Y_M$.   By reorienting $Q$ if necessary we may assume that $n_M<0$.  

Consider the triple $(Q, 0, -e_M^2)$.  This admits a $\cs$-partition with  $\partial_\gl Q= |n_M|(-Y_M)$, by our assumption that $-e_M^2<\tau(\pm Y_i)$ for $i<M$. Let $X$ be the 4-manifold obtained by gluing
 $W_M$ and $(|n_M|-1) W_M'$ to $Q$ along  $\partial_\gl Q$.  Let $\tilde e\in H^2(X)$ be the unique class whose restriction to  $W_M$ equals $e_M$ and whose restriction to $Q$ and each copy of  $W_M'$  equals zero.  
 
Corollary \ref{cor:defcob2} and induction then shows that $(X, \tilde e)$ satisfies Property \inst\ and $(X, \tilde e, -e_M^2)$ admits a $\cs$-partition with $\partial_\gl X$ empty.   But   since $(X,\tilde e)$ satisfies Property \inst, 
$\calm(X,\tilde e)$ has an odd number of singular points, contradicting Proposition \ref{prop:compactness2}.
\end{proof}

\section{The topology of branched covers of Whitehead doubles}\label{2fold}

Theorem $1$ is proved by using the instanton cobordism obstruction to show that 2-fold branched covers of linear combinations of Whitehead doubles cannot bound $\Z/2$-homology balls.  To apply the obstruction, it will be useful to understand the 2-fold branched covers of Whitehead doubles from a topological perspective.   We begin by recalling some background and establishing notation.   

\subsection{Preliminaries and Notation}
\label{subsec:prelim}

 We will often work in the context of oriented manifolds with boundary.  Boundaries of all manifolds will be oriented by the outward normal first convention.   

Given an oriented knot $K\subset S^3$, let $\doub(K)$ denote the {\em positive-clasped untwisted Whitehead double} of $K$.  This is a satellite knot of $K$, defined as follows. Consider the Whitehead link $A\cup C\subset S^3$, shown in Figure \ref{fig1}.   The component labeled $A$ is unknotted, and thus its complement $ S^3\setminus n(A)$ (we use $n(K)$ throughout to denote an open tubular neighborhood of  $K$) is homeomorphic to a solid torus $D^2 \times S^1$ in such a way that the meridian of $A$ corresponds to $\{1\}\times S^1$ and the longitude of $A$ corresponds to $\partial D^2\times \{1\}$.  Identify this solid torus with a neighborhood of $K$ by a homeomorphism which takes  the meridian of $A$  to the (Seifert) longitude for $K$  and the longitude of $A$ to the meridian for $K$.   Under this homeomorphism, the image of $C$ becomes a knot $\doub(K)\subset S^3$, the positive-clasped untwisted Whitehead double of $K$. 

\begin{figure}
\psfrag{A}{$A$}
\psfrag{C}{$C$}
\begin{center}
 \includegraphics [height=110pt]{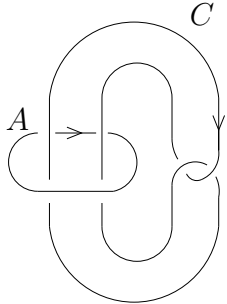}
\caption{\label{fig1} The Whitehead link. Identifying $S^3\setminus n(A)$ with $n(K)$ allows us to view $C$ as a satellite knot living in $n(K)$.  This satellite is the Whitehead double $D(K)$.}
\end{center}
\end{figure}

As mentioned in the introduction, the set of oriented knots modulo the concordance relation forms an abelian group  $\mathcal{C}$.  Addition in $\mathcal{C}$ is played by connected sum of knots, and the inverse operation corresponds to taking the reverse mirror image; that is, in a given projection of $K$ we simultaneously change all crossings (from over to under and vice versa) and reverse the orientation of $K$ as a 1-manifold. 
In light of this, we let $-K$ denote the reverse mirror image of $K$ and $K_1+K_2$ denote the connected sum of $K_1$ with $K_2$.  Thus $K_1-K_2$ means $K_1+(-K_2)$.

In a similar vein, we use the notation $M_1+M_2$ for the oriented connected sum of oriented 3-manifolds, and $-M$ for the manifold $M$ with its orientation reversed,  so that  $M_1-M_2=M_1+(-M_2)$.     The notation is motivated by the fact that 3-manifolds, like knots, can be endowed with group structures if we take into account various cobordism relations.  For our purposes,  two oriented 3-manifolds, $M_1,M_2,$  satisfying $H_*(M_i;\Z/2)\cong H_*(S^3;\Z/2)$ are called {\em $\Z/2$-homology cobordant} if  there is a smooth oriented 4-manifold $Q$ satisfying $\partial Q= -M_1 \sqcup M_2$ and $H_*(Q;\Z/2)\cong H_*(S^3\times I;\Z/2)$.  The set of $\Z/2$-homology cobordism classes form a group called the {\em $\Z/2$-homology cobordism group}, denoted $\Theta_{\Z/2}$.

We will use $\branch (K)$ to denote the   2-fold branched cover of $S^3$, branched over $K$.  Since the Alexander polynomial of  $\doub(K)$ equals 1, every cyclic branched cover of  $\doub(K)$  is an integer homology 3-sphere and, in particular, gives rise to an element in $\Theta_{\Z/2}$.

\medskip

 The following well-known lemma  obstructs sliceness by studying the $\Z/2$-homology cobordism class of a knot's 2-fold branched cover.  For a proof, see \cite[Lemma 2]{cassongordon}.

\begin{lemma}\label{branchslice}
Let $K\subset S^3$ be a knot.  Then $\branch(K)$ is a $\Z/2$ homology 3-sphere; that is, $H_*(\branch(K);\Z/2)\cong H_*(S^3;\Z/2)$.   If $K$ is slice, then $\branch (K)=\partial Q$, where $Q$ is a $\Z/2$-homology 4-ball,  $H_*(Q;\Z/2)\cong H_*(B^4;\Z/2)$.\qed
\end{lemma}

Note that $\branch (-K)=-\branch (K)$, and $\branch (K_1+ K_2)=\branch (K_1)+\branch (K_2)$ (for the latter statement, the reducing sphere in the branched cover is obtained as the branched cover of the sphere used in the non-ambient definition of $K_1+K_2$).  Together with the lemma, these observations show that the 2-fold branched cover operation induces a homomorphism:
$$ \Sigma:  \mathcal{C}\rightarrow \Theta_{\Z/2}.$$
As mentioned, we will prove Theorem $1$ by showing that the image in $\Theta_{\Z/2}$ of linear combinations of Whitehead doubles are not zero.  For this we have the following corollary of Theorem \ref{thm:indcriterion}.  

  \begin{corollary}\label{cor:indknot}
  Let $\{K_i\}_{i=1}^M$ be a set of knots.  For each $i$, suppose there is  a pair $(W_i,e_i)$   with Property \inst  and a negative definite 4-manifold $W_i'$, so  that    $(W_i,e_i,-e_i^2)$   and $(W_i',0,-e_i^2)$  admit   $\cs$-partitions with $\Sigma(K_i)=\partial_\gl W_i=\partial_\gl W'_i$.    If $$  -e_i^2 < \tau(\pm \Sigma(K_j))  \ \ \mathrm{for \ all\ \ }  j<i,$$
then $\{K_i\}_{i=1}^M$ is independent in $\mathcal{C}$.
\end{corollary}

 \begin{proof}  If $n_1K_1\#\cdots \# n_MK_M$ is slice for some $n_i\in \Z$, then Lemma \ref{branchslice} shows that $$\Sigma(n_1K_1\#\cdots \# n_MK_M)\cong n_1\Sigma(K_1)\#\cdots\# n_M(K_M)=\partial Q,$$ where $Q$ is a $\Z/2$-homology ball.  There is a natural cobordism from the connected sum of 3-manifolds to the disjoint union, obtained by attaching 3-handles along the reducing 2-spheres of the prime decomposition.  Attaching this cobordism to the boundary of $Q$  results in a $4$-manifold with the $\Z/2$-homology of a punctured ball, and whose boundary is $$n_1\Sigma(K_1)\sqcup \cdots \sqcup n_M(K_M).$$ Theorem \ref{thm:indcriterion} yields the result.
 \end{proof}

\subsection{Decomposition of $\branch(\doub(K))$ and presentation of $\pi_1$}

This subsection has two purposes.  The first is to obtain a description of the 2-fold branched cover of the Whitehead double of $K$ as a union of three simple pieces: two copies of $S^3\setminus n(K)$, together with the complement of the $(2,4)$ torus link.    This decomposition will be used throughout the paper.   With an eye towards calculating the Chern-Simons invariants, the second aim is to find a suitable presentation of the fundamental group of the $(2,4)$ torus link complement.

 To begin our analysis, note that $\doub(K)$ can be described in a slightly different way.  Namely, we can construct $\doub(K)$ by replacing a tubular neighborhood of the $A$ component of the Whitehead link by 
$S^3\setminus n(K)$.  After this identification, the $C$ component represents $\doub(K)$.  To make this precise, we must specify the way that boundary tori are identified.   To this end, denote the meridian  and longitude  of $K$ by $\mu_K$ and  $\lambda_K$, respectively.  Similarly, denote the  meridian and longitude of the $A$ component of the Whitehead link by $\mu_A$ and $\lambda_A$.  Then the union 
$$ S^3\setminus n(K) \underset{{(\mu_K,\lambda_K)=(\lambda_A,\mu_A)}}\cup S^3\setminus n(A) $$
is homeomorphic to the 3-sphere, and the image of $C$ under this identification is equivalent to $\doub(K)$.  The notation is meant to indicate that the two manifolds are glued by a diffeomorphism of their boundary tori which interchanges meridian-longitude pairs.  Using this description of the Whitehead double, we can easily obtain the following decomposition of its 2-fold branched cover.

 \begin{prop} \label{2foldcover} Given a knot $K\subset S^3$, the 2-fold branched cover $\branch (\doub(K))$ of $S^3$ branched over $\doub(K)$  has a decomposition (see Figure \ref{fig2})
$$\branch (\doub(K))= X_1\underset{T_1}\cup Y\underset{T_2}\cup  X_2,$$
where 
\begin{enumerate}
\item $X_i= S^3\setminus n(K)$, for $i=1,2$. The framing of $\partial X_i$ is inherited from the standard meridian-longitude for  $K$, and denoted $\mu_{K_i}, \lambda_{K_i}$.
\item $Y=S^3\setminus n(A_1\cup A_2)$, where $A_1\cup A_2\subset S^3$ is the $(2,4)$ torus link, as in Figure \ref{fig2}.  
\item $T_i$ is oriented so that $\mu_{K_i}\cdot \lambda_{K_i}=1$, $i=1,2$.
\item  $X_1,Y,$ and $X_2$ are oriented so that $\partial X_1=-T_1 $, $\partial X_2=-T_2$, $\partial Y=T_1\sqcup T_2$.  
\item  The gluing identifications $\partial X_i\subset \partial Y$ are specified by 
$$\mu_{K_i}=\mu_{A_i}^{-2}\lambda_{A_i}\text{ and }\lambda_{K_i}=\mu_{A_i}\text{ for }i=1,2,$$
where  $\mu_{A_i},\lambda_{A_i}, i=1,2$ denote the  standard meridian-longitude pairs for  $A_1\cup A_2$.  
\end{enumerate}

\end{prop}

\begin{proof}

First, note that the Whitehead link admits an isotopy that interchanges its components.  Figure \ref{isotope} applies this isotopy to the description of the Whitehead double which preceded the proposition.   Note that  $\lambda_A$ appears to acquire two additional meridional twists during the isotopy; these are added to cancel out the twists coming from the writhe of the new projection of $A$.

It is now clear that our description of the 2-fold branched cover of $ D(K)$ holds.  Indeed, the manifold presented by Figure \ref{fig2} has an obvious $\Z/2$-symmetry; namely, the rotation by $\pi$ about the axis which is perpendicular to the plane of the page and which passes through the central marked point.  Forming the quotient of this action produces the manifold shown in the second part of Figure \ref{isotope}, with $C$ representing the image of the axis (together with the point at infinity) under the quotient projection.  Since the axis coincides with the fixed set of the rotation,  Figure \ref{fig2} presents the 2-fold branched cover of Figure \ref{isotope} branched over $C$ which, by our previous description, is equivalent to  $D(K)\subset S^3$. 
\end{proof}

\begin{figure}[h]
 \psfrag{A}{ $ \mu_{A_2}=\lambda_{K_2} $}
\psfrag{B}{\hskip.05in$ \mu_{K_2}$} 
\psfrag{C}{$A_2$}
\psfrag{D}{$A_1$}
\psfrag{E}{$\mu_{A_1}=\lambda_{K_1}$ }
\psfrag{F}{$\mu_{K_1}$}
\psfrag{p}{$\pi$}
\begin{center}
  \includegraphics[height=180pt]{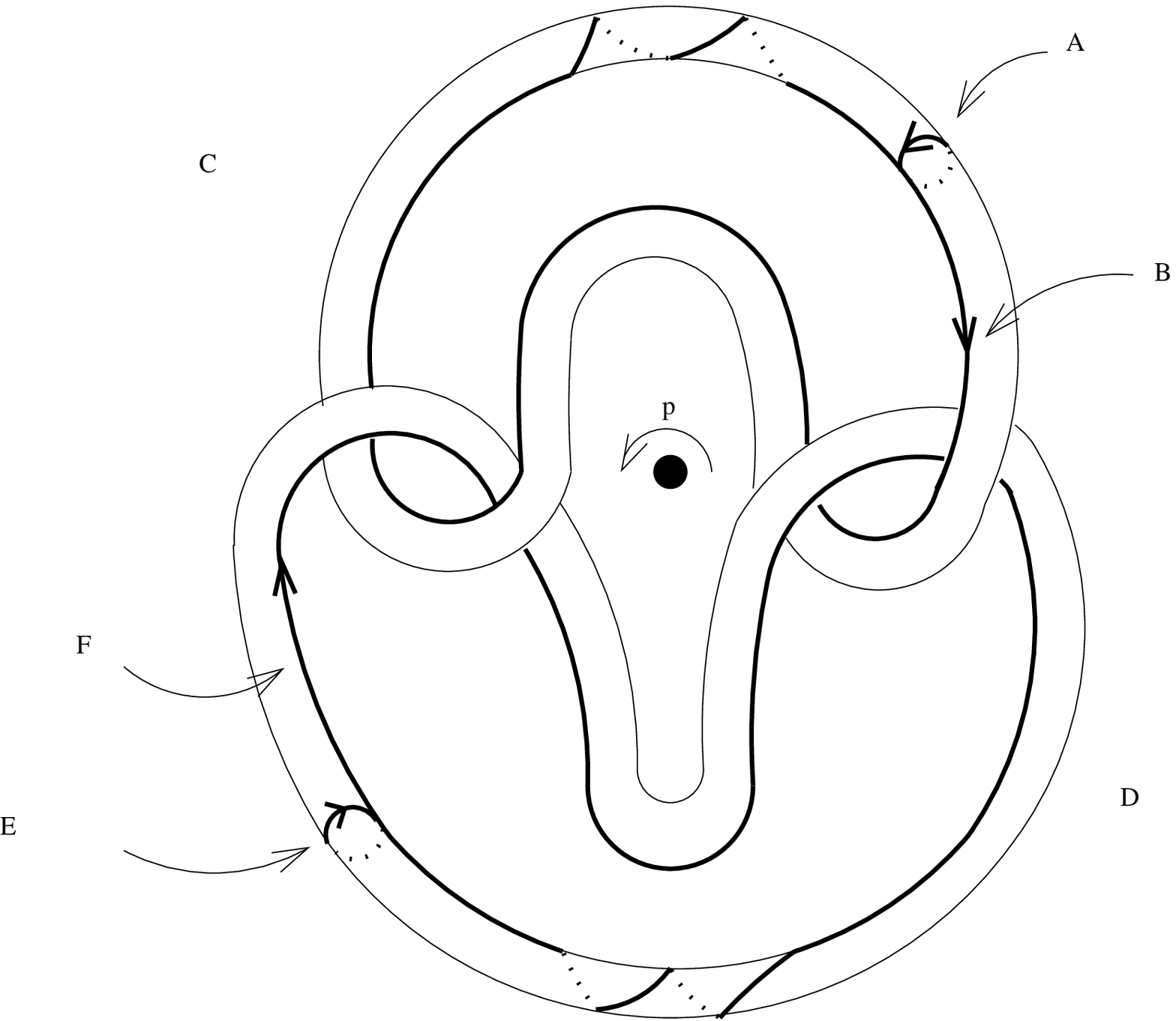}
\caption{\label{fig2} }
\end{center}
\end{figure}

\begin{figure}[h]
 \psfrag{m}{ $ \mu_{A}=\lambda_{K} $}
\psfrag{l}{$\lambda_A=\mu_{K}$} 
\psfrag{A}{$A$}
\psfrag{C}{$C$}
\psfrag{e}{$\sim$ }
\begin{center}
  \includegraphics[height=160pt]{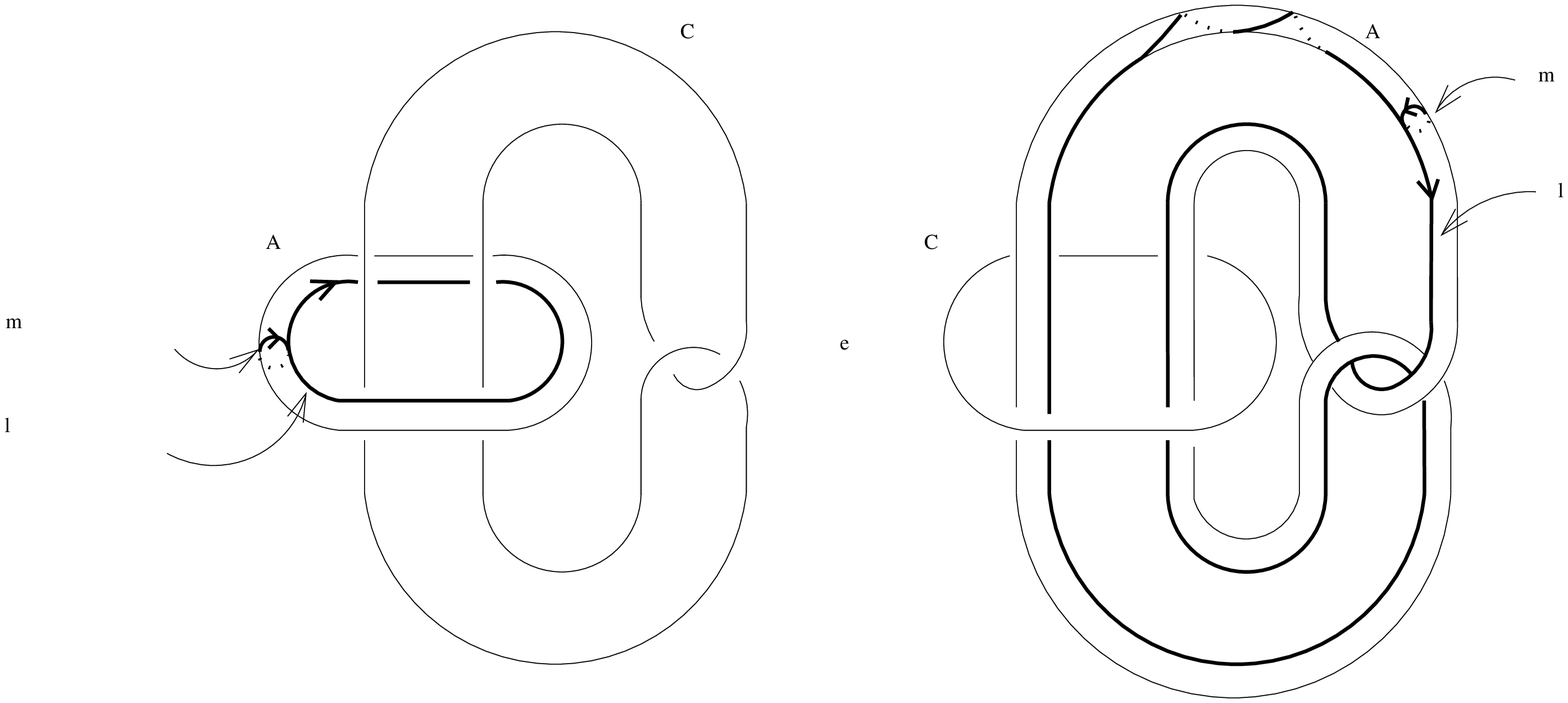}
\caption{\label{isotope} }
\end{center}
\end{figure}

 When analyzing moduli spaces of flat $SO(3)$ connections, we will need a presentation of the fundamental group of $Y=S^3\setminus n(A_1\cup A_2)$. Recall that loops in a knot or link diagram are based by connecting them to a base point lying above the plane of the projection by a straight line segment. In this way the loops $\mu_{A_i},\lambda_{A_i}$ in Figure \ref{fig2} represent elements of $\pi_1(Y)$. Eliminating generators from  the Wirtinger  presentation yields the following.

\begin{prop} \label{pi1} The meridians $\mu_{A_1}$ and $\mu_{A_2}$  generate the fundamental group of  $Y$. Indeed, we have a presentation:
$$\pi_1(Y)=\langle\ \mu_{A_1}, \mu_{A_2} \ | \ [\mu_{A_1}^{-1},\mu_{A_2}][\mu_{A_1},\mu_{A_2}^{-1}] \ \rangle.
$$
The longitudes can be described in this presentation by
$$\lambda_{A_1}=\mu_{A_1}\mu_{A_2}\mu_{A_1}^{-1}\mu_{A_2}, \ \ \ \lambda_{A_2}=\mu_{A_2}\mu_{A_1}\mu_{A_2}^{-1}\mu_{A_1}. $$
  Moreover, the loop $(\mu_{A_1}\mu_{A_2}^{-1})^2$  lies in the center of $\pi_1(Y)$.\qed
\end{prop}

\subsection{Some useful cobordisms}
 
In this subsection, we complete our topological understanding of the  2-fold branched covers by constructing negative definite cobordisms from  $\Sigma(D(T_{p,q}))$ to certain Seifert manifolds, $\Sigma(p,q,2pq-1)$.   The motivation underlying this construction is provided by the fact that these Seifert manifolds bound 4-manifolds with Property \inst.  According to Proposition \ref{prop:defcob}, gluing these latter 4-manifolds to the cobordisms provided here will verify that $\Sigma(D(T_{p,q}))$ bounds a 4-manifold with Property \inst.

\medskip

We begin with a useful lemma which is undoubtedly familiar to those working with the Kirby calculus (similar observations can be found, for instance, in \cite{Cochran-Gompf}).  In essence, the lemma says that performing crossing changes to a knot leads to a cobordism (rel boundary) between knot complements. Keeping track of the sign of the crossings determines the intersection form of these cobordisms.

\begin{figure}
 \psfrag{K-}{negative}
  \psfrag{K+}{positive}
 \begin{center}
 \includegraphics [height=100pt]{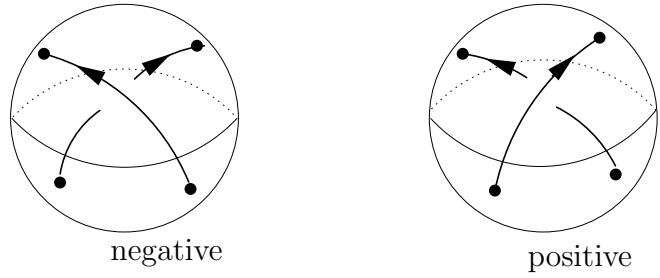}
\caption{\label{fig:crossingchange} Changing a crossing.  Outside a ball, two knots agree.  Making a modification in this ball which passes from left to right is called a negative-to-positive crossing change. }
\end{center}
\end{figure}

\begin{lemma}\label{lem3.2}  Let  $K_1,K_2$ be oriented knots in an oriented  homology 3-sphere $\Sigma$. Suppose that $K_1$ can be transformed to $K_2$ by  $p$ positive-to-negative  and $n$ negative-to-positive  crossing changes (see Figure \ref{fig:crossingchange}). 

 Then adding 2-handles to $[0,1]\times (\Sigma\setminus n(K_1))$ produces  a cobordism rel boundary from $\Sigma\setminus n(K_1)$ to $\Sigma\setminus n(K_2)$; that is,  an oriented 4-manifold $V$ with boundary decomposing as
$$\partial V=-\big(\Sigma\setminus n(K_1)\big)  \underset{\{0\} \times T}\cup \big( [0,1]\times T \big) \underset{ \{1\}\times T}\cup\big(\Sigma\setminus n(K_2)\big),$$where $T$  is the oriented boundary of the neighborhood of $K_1$.  This cobordism satisfies
\begin{enumerate}
\item There is an isomorphism $H_1(V)\cong\Z$ which extends  the 
linking number $H_1(\Sigma\setminus n(K_i))\to \Z$  for $i=1,2$.
\item Let $H$ be a 3-manifold satisfying $H_*(H;\Q)\cong H_*(S^1\times D^2;\Q)$ (and hence  $\partial H\cong T$). 
Then any 4-manifold  obtained by gluing $[0,1]\times H $ to $V$ along $[0,1]\times T \subset \partial V$ so that the two boundary components are rational homology 3-spheres has intersection form $\langle-1\rangle^{\oplus p} \oplus \langle1\rangle^{\oplus n}$.  \end{enumerate}

 \end{lemma}

\begin{proof}  Attach   2-handles  to $[0,1]\times \Sigma $ along  unknotted $-1$ framed circles  in $ \{1\}\times \Sigma $ which encircle each crossing one wishes to change from positive to negative, and attach 2-handles  to $[0,1]\times \Sigma $ along  unknotted $+1$ framed circles   which encircle each crossing one wishes to change from  negative to positive.  Choose the attaching circles to have linking number zero 
with $K_1$.  Since the handles are attached along unknots, the resulting 4-manifold $W$ satisfies $\partial W=-\Sigma\sqcup \Sigma$.  As $K_1$ lies in the complement of the attaching regions for the 2-handles, its trace $ [0,1]\times  K_1 $ is  a properly embedded annulus in $ W$ which we denote by $F$.  Let $V=W\setminus n(F)$.

Now $W$ is diffeomorphic to the connected sum of $[0,1]\times  \Sigma $ with $p$ copies of $- {\C P}^2$ and $n$ copies of $\C P^2$. Hence $H_1(W)=0$ and $W$ has   diagonal intersection form
$\langle-1\rangle^{\oplus p} \oplus \langle1\rangle^{\oplus n}$.   Sliding $K_1$ over the 2-handles shows that  
$$\partial (W,F)= -(\Sigma,K_1) \sqcup (\Sigma,K_2).$$
Thus $V$ is a cobordism rel boundary from $\Sigma\setminus n(K_1)$ to $\Sigma\setminus n(K_2)$.

 The condition that the attaching circles for the 2-handles have linking number zero with $K_1$ ensures that the annulus  $F$ is null-homologous; that is, $[F,\partial F]=0 \in H_2(W,\partial W)\cong H_2(W)$.  Thus the linking number $H_1(V)=H_1(W\setminus n(F))\to \Z$ extends the linking number $H_1(\Sigma\setminus n(K_i))\to \Z$ for $i=1,2$. Since $H_1(W)=0$,   the linking number  yields an isomorphism $H_1(V)\cong \Z$.

The homomorphism $H_2(W,V)\to H_1(V)$ is an isomorphism and $H_3(W,V)\to H_2(V)$ is zero. This implies $H_2(V)\cong H_2(W)$. Since the generators of $H_2(W)$ can be represented by cycles which miss $F$ (which, again, follows from the fact the attaching circles for the 2-handles have zero linking with $K_1$)   it follows that  $V$, like $W$, has intersection form  $\langle-1\rangle^{\oplus p} \oplus \langle1\rangle^{\oplus n}$.   Gluing $[0,1]\times  H $ to $V$ along $F\times S^1 =[0,1]\times  T $ yields a 4-manifold with intersection form $\langle-1\rangle^{\oplus p} \oplus \langle1\rangle^{\oplus n}$, provided that this gluing does not create any additional second homology.  However, the assumption that the boundary components of the glued manifold are rational homology spheres implies  $H_2(V)= H_2(([0,1]\times H)\!\!\! \underset{[0,1]\times T}{\cup}\!\!\!  V)$. 
\end{proof}

Recall that  our decomposition of the 2-fold branched cover of $D(K)$ contains two copies of $S^3\setminus n(K)$.  The previous lemma allows us to construct cobordisms from these knot complements to a solid torus (the complement of the unknot) which,  in turn, yield cobordisms from $D(K)$ to simpler manifolds.  If $K$ can be unknotted by  changing positive crossings, these cobordisms will be negative definite. The following lemma makes this precise.  In the statement, $S^3_{\frac{p}{q}}(K)$ denotes the result of  Dehn surgery on $K$ with slope $\frac{p}{q}$.

\begin{lemma}\label{lem3.5}Suppose  $K$ is a knot which can be unknotted by positive-to-negative crossing changes. Then there is
   a negative definite 4-manifold   
 $N(K)$ with $H_1(N(K))=0$, and
 $$\partial N(K)=-\branch (\doub(K))\sqcup  \surg{1}{2}{K}.$$ 
\end{lemma}

\begin{proof}

The hypothesis allows us to use Lemma \ref{lem3.2} to produce a cobordism rel boundary $V(K)$ between $S^3\setminus n(K)$ and $S^3\setminus n(U)$.  Thus $$\partial V(K)=-\big(S^3\setminus n(K)\big)  \underset{\{0\} \times T}\cup \big( [0,1]\times T \big) \underset{ \{1\}\times T}\cup\big(S^3\setminus n(U)\big).$$
Since only positive crossings were changed, $V(K)$ is negative definite.   

Recall the decomposition of Proposition \ref{2foldcover}  (c.f.~ Figure \ref{fig2})
 $$\branch (\doub(K))=X_1\underset{T_1} \cup Y\underset{T_2}\cup  X_2,$$
 with $X_i= S^3\setminus n(K)$ and gluing identifications on the boundary tori specified by $\mu_{K_i}=\mu_{A_i}^{-2}\lambda_{A_i}$ and $\lambda_{K_i}=\mu_{A_i}$.  In light of this decomposition, let  $$H=X_1\underset{T_1}\cup Y.$$
Gluing $[0,1]\times H$ to $V(K)$ produces a negative definite 4-manifold $N(K)$ with boundary components
 $$\partial_-(N(K))= -\big( H  \underset{T_2} \cup S^3\setminus n(K)\big)=-\big( H  \underset{T_2} \cup X_2\big)=-\branch (\doub(K))$$ and
 $$\partial_+(N(K))=H \underset{T_2}\cup (S^3\setminus n(U)).$$ 

Since $S^3\setminus n(U)$ is a solid torus, 
 $ \partial_+(N(K)) $
is obtained from $H$ by Dehn filling  $\partial H=T_2$ or, equivalently, from  $X_1\cup_{T_1}(S^3\setminus n(A_1))$ by Dehn surgery on $A_2$.  This is the {\em trivial} slope (i.e.  $\frac{1}{0}$) Dehn surgery, since the longitude $\lambda_U$ bounds a disk in $S^3\setminus n(U)$  and $\lambda_U$ is identified with $\mu_{A_2}$. 

Therefore, 
$$ \partial_+(N(K))=X_1\underset{T_1} \cup (S^3\setminus n(A_1)).$$
Note, however, that $A_1$ is also unknotted so $S^3\setminus n(A_1)$ is a solid torus. Since $X_1=S^3\setminus n(K)$, it follows that $\partial_+(N(K))$ is a Dehn surgery on $K$.  To understand which surgery, observe that the longitude $\lambda_{A_1}$ bounds a disk in $S^3\setminus n(A_1)$, so the Dehn surgery is the one for which $\lambda_{A_1}=\mu_{A_1}^2\mu_{K_1}= \lambda_{K_1}^2\mu_{K_1}$ is killed, i.e. 
$$\partial_+(N(K))=\surg{1}{2}{K} .$$ 
Thus  $\partial N(K)=-\branch (\doub(K))\sqcup \surg{1}{2}{K} $, as claimed.  

Examining  the construction in the proof of Lemma \ref{lem3.2} shows that $N(K)$ is obtained from $I\times \branch (\doub(K))$ by attaching 2-handles, and hence  $H_1(N(K))=0$. 
 \end{proof}

\medskip

Let $p,q$ be a pair of relatively prime positive integers, and denote by $T_{p,\pm q}$   the $(p,\pm q)$ torus knot.   For $k>0$,    the Seifert fibered homology sphere
 $-\Sigma(p,q, pqk \mp 1)$ is diffeomorphic to $\surg{1}{k}{T_{p,\pm q}}$  as oriented manifolds \cite[Proposition 3.1]{Moser1971}.  In particular, $\surg{1}{k}{T_{p,q}} =-\Sigma(p,q,kpq-1) $ for $k>0$.
Since $T_{p,q}$ has a projection with all positive crossings, Lemma \ref{lem3.5} yields the cobordisms promised at the beginning of this subsection as an immediate corollary.

\begin{corollary} \label{prop5.1}  Given $p,q>0$ relatively prime, there exists a pair  $(W_{p,q},e)$ with Property \inst, with $\partial W_{p,q}$  the disjoint union of the (oppositely oriented) 2-fold branched cover of the Whitehead double of  $T_{p,q}$   and three lens spaces $Y_i$
$$\partial W_{p,q}= -\Sigma(D(T_{p,q}))\sqcup  Y_1\sqcup Y_2\sqcup Y_3.$$
The class $e$ satisfies  $-e^2=\tfrac{1}{pq(2pq-1)},$  and the triples $(W_{p,q}, e, \tfrac{1}{pq(2pq-1)}) $ and  $(W_{p,q}, 0, \tfrac{1}{pq(2pq-1)}) $ admit a $\cs$-partition with $\partial_\gl W_{p,q}= -\Sigma(D(K))$.
\end{corollary}

\begin{proof} According to Proposition \ref{truncated}, there is a pair $(X,e)$ satisfying Property \inst\ such that the boundary of $X$ is the disjoint union of $\Sigma(p,q,2pq-1) $ and three lens spaces 
$Y_1\sqcup Y_2\sqcup Y_3$.  Moreover, the lens spaces satisfy the Chern-Simons bounds $$\tau(Y_i,e_i)>-e^2=\tfrac{1}{pq(2pq-1)}, $$ for any classes $e_i\in H^2(Y_i)$.
 Thus $(X,e,-e^2)$ and $(X,0,-e^2)$ admit $\cs$-partitions with $$\partial_\gl X= \Sigma(p,q,2pq-1) \ \ \ \ \ \  \partial_\cs X=Y_1\sqcup Y_2\sqcup Y_3.$$  
 
 Let $N(T_{p,q})$ be the negative definite manifold of Lemma \ref{lem3.5}, whose boundary is
 $$\partial N(T_{p,q})= -\Sigma(D(T_{p,q}))\sqcup S^3_{\frac{1}{2}}(T_{p,q})=-\Sigma(D(T_{p,q}))\sqcup -\Sigma(p,q,2pq-1).$$   Gluing $N(T_{p,q})$ to $X$ yields a negative definite 4-manifold $W_{p,q}$, with boundary 
 $$\partial W_{p,q} = -\Sigma(D(T_{p,q}))\sqcup  Y_1\sqcup Y_2\sqcup Y_3.$$
 Extend the class $e$ by zero over $N(T_{p,q})$. Then  Corollary \ref{cor:defcob2} implies that $(W_{p,q},e)$ satisfies Property \inst, and that  the triples $(W_{p,q}, e, \tfrac{1}{pq(2pq-1)}) $ and  $(W_{p,q}, 0, \tfrac{1}{pq(2pq-1)}) $ admit a $\cs$-partition with $\partial_\gl W_{p,q}= -\Sigma(D(K))$.
 \end{proof}

 \section{ Chern-Simons invariants and the proof of Theorem $1$} \label{section:proof}
 
 In this section we bring together our newfound topological understanding of the $2$-fold branched covers of  Whitehead doubles with the instanton obstruction developed in Section \ref{gauge}.    Before beginning, we observe that  
 Proposition \ref{infinite2}  and  Corollary \ref{prop5.1} have an immediate consequence; namely, that the Whitehead double  of $T_{p,q}$ has infinite order in the smooth knot concordance group when $p,q>0$.   This result can be seen by a variety of methods (see, for instance, nearly all the references mentioned in the introduction).    Proving Theorem 1, however, is clearly a big leap from this result.  It involves studying the interaction between different Whitehead doubles in $\mathcal{C}$.  To accomplish this, we would like to use our independence criterion (Theorem \ref{thm:indcriterion}) or, more precisely, its corollary (Corollary \ref{cor:indknot}).   Since Corollary \ref{prop5.1} has already provided the requisite $4$-manifolds, it remains to understand  the Chern-Simons invariants of $\Sigma(D(T_{p,q}))$.  The present section develops such an understanding.

\subsection{Chern-Simons invariants of $\branch (\doub(K))$} 

We begin in some generality, analyzing the flat moduli spaces and Chern-Simons invariants of $\branch (\doub(K))$ for an arbitrary knot $K$.  In the next subsection we apply this analysis to the torus knots at hand.

\medskip

For any finite CW complex $X$ and compact Lie group $G$, let $\chi_G(X)$ denote the space of conjugacy classes of representations $\pi_1(X)\to G$.

For a knot $K\subset S^3$, the moduli space of flat $SO(3)$ connections on $\branch (\doub(K))$ is homeomorphic (in fact analytically isomorphic \cite{FKK}) to  $\chi_{SO(3)}(\branch (\doub(K)))$.   
Given any $\Z/2$--homology sphere $\Sigma$, a simple obstruction theory argument shows that  the 2-fold cover $SU(2)\to SO(3)$ induces a bijection between the space $\chi_{SU(2)}(\Sigma)$ of conjugacy classes of $SU(2)$  representations of $\pi_1(\Sigma)$ and the space $\chi_{SO(3)}(\Sigma)$ of conjugacy classes of $SO(3)$  representations.    Indeed, the obstruction to lifting a representation $\gamma:\pi_1(\Sigma)\to SO(3)$ to $SU(2)$ lies in $H^2(\Sigma;\Z/2)$, which vanishes since $\Sigma$ is a $\Z/2$-homology sphere.  A similar argument using the fact that $H^1(\Sigma;\Z/2)=0$ shows that the lift is unique.

Since $\branch (\doub(K))$ is a homology 3-sphere, 
$$\chi_{SU(2)}(\branch (\doub(K)))=\chi_{SO(3)}(\branch (\doub(K))).$$ 
 
The map $BSU(2)\to BSO(3)$ induced by the 2-fold cover $SU(2)\to SO(3)$ takes the first Pontryagin class $p_1\in H^4(BSO(3))$ to $-4c_2\in H^4(BSU(2))$, where $c_2$ denotes the second Chern class.  More generally, the corresponding identification of Lie algebras $a:\mathfrak{su}(2)\to \mathfrak{so}(3)$ satisfies
 $\Tr_3(a(x)^2)=4 \Tr_2(x^2)$, where $\Tr_n$ denotes the trace on $n\times n$ matrices. It follows that 
 the second Chern form of an $SU(2)$  connection, $c_2(A)=\frac{1}{8\pi^2}\Tr_2(F(A)\wedge F(A))$, and the first Pontryagin form of the corresponding $SO(3)$ connection,  $p_1(a(A))=- \frac{1}{8\pi^2}\Tr_3(a(F(A))\wedge a(F(A)))$, satisfy $c_2(A) =-4 p_1(A)$. Thus for a $\Z/2$-homology sphere, the $SO(3)$ Chern-Simons invariant of a representation $\gamma$ (defined in terms of the first Pontryagin form) and the $SU(2)$ Chern-Simons invariant  (defined in terms of the second Chern form) of its unique lift $\tilde{\gamma}$ are related by $\cs(\gamma)=-4\cs(\tilde{\gamma})$.

These observations allow us to work in the technically and notationally simpler context of $SU(2)$ representations.  We will identify $SU(2)$ with the unit quaternions, so that the diagonal maximal torus corresponds to the subgroup  of unit complex numbers $\{e^{i\theta}\}$.

\medskip 
To ensure that the instanton moduli space is compact requires establishing    Chern-Simons bounds, as explained in Section \ref{gauge}. Obtaining an explicit   description of the moduli space of flat $SO(3)$ connections  on a closed 3-manifold sufficient to estimate all Chern-Simons invariants is a difficult problem in general, but it turns out that for the 2-fold branched cover of the Whitehead double of a knot there are essentially only two kinds of interesting flat connections, as we next explain.

We return once again to the decomposition
of Proposition \ref{2foldcover}   
$$\branch (\doub(K))= X_1\underset{T_1} \cup Y\underset{T_2}\cup X_2,$$
 with the gluing identifications on the boundary tori specified by $\mu_{K_i}=\mu_{A_i}^{-2}\lambda_{A_i}$, and $\lambda_{K_i}=\mu_{A_i}$.

 The space  $ \chi_{SU(2)}(\branch (\doub(K)))$ decomposes  into eight subsets as follows. Given  labels $x_1,y,x_2\in \{A,N\}$ let $\chi_{SU(2)}(\branch (\doub(K)))_{ x_1y x_2}$ denote those conjugacy classes of $SU(2)$ representations of $\pi_1(\branch (\doub(K)))$ whose restriction  to the subgroups $\pi_1(X_1), \pi_1(Y), \pi_1(X_2)$ have image (respectively) 
an abelian or non-abelian subgroup of $SU(2)$.   

More precisely, the inclusion  of $X_1 $ into $\branch (\doub(K))$ induces a homomorphism  on fundamental groups.   Composing this with a representation $\tilde\gamma:\pi_1(\branch (\doub(K)))\to SU(2)$ yields a representation $\pi_1(X_1)\to SU(2)$ whose image is either an abelian or non-abelian subgroup of $SU(2)$.   Changing $\tilde \gamma$ within its conjugacy class  (or changing the base point)  does not change the type of this image.  Thus the label $x_1\in \{A,N\}$ can unambiguously be assigned to the conjugacy class of $\tilde \gamma$.  Similar comments apply to $ Y $ and $ X_2 $.

The following theorem shows that only three of these eight subsets of  $ \chi_{SU(2)}(\branch (\doub(K)))$  are interesting.

\begin{theorem} \label{lem6.1}  Let $K\subset S^3$ be any knot.  Then the four subspaces $$\chi_{SU(2)}(\branch (\doub(K)))_\sss{ANN}\ \ \ \ \ \ \ \ \ \ \ \ \chi_{SU(2)}(\branch (\doub(K)))_{\sss{NNA}}$$ $$ \chi_{SU(2)}(\branch (\doub(K)))_\sss{ANA}  \ \ \ \ \ \ \ \ \ \ \ \ \chi_{SU(2)}(\branch (\doub(K)))_\sss{NNN}$$ 
are empty. The subspace $\chi_{SU(2)}(\branch (\doub(K)))_\sss{AAA}$ contains only the trivial representation.

Thus every non-trivial representation lies in one of
$$\chi_{SU(2)}(\branch (\doub(K)))_\sss{AAN} \ \ \ \ \ \chi_{SU(2)}(\branch (\doub(K)))_\sss{NAA} \ \ \ \ \ \chi_{SU(2)}(\branch (\doub(K)))_\sss{NAN}.$$
The branched covering transformation induces an involution on these three subsets that exchanges the first two and leaves the third invariant.
\end{theorem}
\begin{proof}
We use the presentation for $\pi_1(Y)$ derived in  Proposition \ref{pi1}.   

We first show that the restriction of any $SU(2)$ representation of $\pi_1(\branch (\doub(K)))$ to $\pi_1(Y)$ has image an abelian subgroup of $SU(2)$. 
Suppose to the contrary that there exists a representation $\tilde \gamma:\pi_1(\branch (\doub(K)))\to SU(2)$ whose restriction to $\pi_1(Y)$ has non-abelian image.  The centralizer of any non-abelian subgroup of $SU(2)$ consists only of the center $\{\pm 1\}$.  The loop $(\mu_{A_1}\mu_{A_2}^{-1})^2$ lies in the center of $\pi_1(Y)$, and hence is sent by $\tilde \gamma$ to $\pm 1$.  Thus
$$\tilde  \gamma(\mu_{K_1})=\tilde \gamma(\mu_{A_1}^{-2}\lambda_{A_1})=\tilde \gamma(\mu_{A_1}^{-1}\mu_{A_2}\mu_{A_1}^{-1}\mu_{A_2})
=\tilde \gamma(\mu_{A_1}^{-1}(\mu_{A_1}\mu_{A_2}^{-1})^{-2} \mu_{A_1})=\pm 1.$$

Since $X_1$ is the complement of a knot in $S^3$, $\pi_1(X_1)$ is normally generated by its meridian $\mu_{K_1}$.  Therefore $\tilde \gamma$ sends every loop in $\pi_1(X_1)$ to $\{\pm 1 \}$.   In particular, we have  $$\tilde \gamma(\lambda_{K_1})=\tilde \gamma(\mu_{A_1})\in \{\pm 1 \} $$   Since both  $\pm 1$ are central, the fact that $\pi_1(Y)$ is generated by $\mu_{A_1}$ and $\mu_{A_2}$, implies that the restriction of $\tilde \gamma$ to $\pi_1(Y)$ is abelian. This contradicts our assumption, and hence every representation of $\pi_1(\branch(\doub(K)))$ restricts to an abelian representation of $\pi_1(Y)$.

Now suppose that $\tilde \gamma\in \chi_{SU(2)}(\branch (\doub(K)))_\sss{AAA}$. We must show that $\tilde \gamma$ is the trivial representation.  We may assume, by conjugating $\tilde \gamma$ if needed,  that the 
restriction of $\tilde \gamma$ to $\pi_1(Y)$ takes values in the maximal torus $\{e^{i\theta}\}$. In particular, $\tilde \gamma(\mu_{A_1})$ and $\tilde \gamma(\lambda_{A_1})$ lie in this torus.   
 Thus there exists an angle $\theta_0$ so that $\tilde \gamma(\mu_{A_1}^{-2}\lambda_{A_1})=e^{i\theta_0}$.

Since $\mu_{K_1}=\mu_{A_1}^{-2}\lambda_{A_1}$ generates $H_1(X_1)$ and the restriction of $\tilde \gamma$ to $\pi_1(X_1)$ is abelian, the restriction of $\tilde \gamma$ to $\pi_1(X_1)$ factors through $H_1(X_1)=\Z\langle \mu_{K_1}\rangle $. Thus every loop in $\pi_1(X_1)$ is sent to the subgroup of the maximal torus generated by $e^{i\theta_0}$.
It follows that the restriction $\tilde \gamma:\pi_1(X_1)\to SU(2)$ has image in the (same) maximal torus $\{e^{i\theta}\}$.   An identical argument shows that the restriction of $\tilde \gamma$ to $\pi_1(X_2)$ also takes values in the maximal torus $\{e^{i\theta}\}$, and so $\tilde \gamma$ itself is an abelian representation. Since $\branch (\doub(K))$ is a homology sphere, $\tilde \gamma$ must be the trivial representation. 

The last assertion follows from the observation that the 2-fold branched covering transformation exchanges $X_1$ and $X_2$ and leaves $Y$ invariant.
\end{proof}

 Theorem \ref{lem6.1} now implies  the following.
 
\begin{corollary}\label{cor6.2}  $\tau(\branch (\doub(K))$ equals the minimum 
of 4 and the values of the two functions
$$-4\cs:\chi_{SU(2)}(\branch (\doub(K)))_\sss{NAA}\to (0,4]$$
and
$$-4\cs:\chi_{SU(2)}(\branch (\doub(K)))_\sss{NAN}\to (0,4]$$
 \end{corollary}
\begin{proof}
The  non-trivial 2-fold branched covering transformation  is an orientation preserving  diffeomorphism of $\branch (\doub(K))$ which  interchanges the subspaces  $\chi_{SU(2)}(\branch (\doub(K)))_\sss{AAN}$ and $\chi_{SU(2)}(\branch (\doub(K)))_\sss{NAA}$. Hence  the  Chern-Simons values on $\chi_{SU(2)}(\branch (\doub(K)))_\sss{AAN}$ and $\chi_{SU(2)}(\branch (\doub(K)))_\sss{NAA}$ are the same.  

Recall that $\tau(\branch (\doub(K))$ denotes the minimal $SO(3)$ Chern-Simons invariant, taken in $(0,4]$.  Since the trivial connection has Chern-Simons invariant  4, the assertion follows. 
\end{proof}
 
 Let $\theta$ denote the trivial representation.

 \begin{prop}\label{lem6.3}  There is a   bijective   correspondence
 between  the flat moduli spaces $$\chi_{SU(2)}(\branch (\doub(K)))_\sss{NAA}\ \ \longleftrightarrow  \ \ \chi_{SU(2)}(\surg{1}{2}{K} )\setminus \{\theta\}.$$ This correspondence preserves the Chern-Simons invariants modulo $\Z$. 
 \end{prop}
 \begin{proof}    
 Let $U\subset S^3$ denote the unknot.  Use Lemma \ref{lem3.2} to construct a 4-dimensional cobordism rel boundary  $V$ from $S^3\setminus n(K_2)$ to $S^3\setminus n(U)= S^1\times D^2$ such that the linking number $H_1(S^3\setminus n(K_2))\to \Z$ extends to an isomorphism $H_1(V)\to \Z$.  
 
 Since any abelian representation factors through  first homology, every abelian $SU(2)$ representation of $\pi_1(X_2)$ extends uniquely to $\pi_1(V)$ and, symmetrically, every (necessarily abelian)    $SU(2)$ representation of $\pi_1(S^1\times D^2)$ extends uniquely to $\pi_1(V)$.  In other words, extending over $V$ sets up a 1-1 correspondence between  abelian representations  of  $\pi_1(X_2)$ 
and  $\pi_1(S^1\times D^2)$.

Glue $[0,1]\times Y$ to $V$  to obtain a cobordism rel boundary $W$ from $Y\underset{T_2} \cup X_2$ to $Y\underset{T_2} \cup (S^1\times D^2)$
 $$W= ([0,1]\times Y )\underset{[0,1]\times T_2 } \cup V.$$   Since $\pi_1([0,1]\times Y )=\pi_1(Y)$,  extending over $W$ sets up a 1-1 correspondence between   $SU(2)$ representations of $\pi_1 (Y\cup_{T_2}X_2)$ whose restriction to $\pi_1(X_2)$ are abelian and  $SU(2)$ representations of $\pi_1(Y\cup_{T_2} (S^1\times D^2))$.   
Notice that $Y\cup_{T_2} (S^1\times D^2)$ is again a solid torus, so that any $SU(2)$ representation of 
$\pi_1(Y\cup_{T_2} (S^1\times D^2))$ is abelian, and hence  its unique extension to $\pi_1(W)$ is also abelian.

Repeat  the construction, gluing $[0,1]\times X_1 $ to $W$  to obtain a cobordism  
$$N=([0,1]\times X_1 )\underset{[0,1]\times T_1 } \cup W = ([0,1]\times X_1 ) \underset{[0,1]\times T_1 } \cup  ([0,1]\times Y ) \underset{[0,1]\times T_2 }  \cup V$$
with two boundary components,  
$$\partial N_-=-\big(X_1\underset{ T_1 } \cup  Y\underset{ T_2 } \cup X_2 \big) =-\branch (\doub(K))$$ and 
$$\partial N_+=X_1\underset{ T_1 } \cup  Y\underset{ T_2 } \cup (S^1\times D^2)= X_1\underset{ T_1 } \cup  (S^1\times D^2)=\surg{1}{2}{K}.$$
As before, extending over $N$ sets up a 1-1 correspondence between representations of $\pi_1(\branch (\doub(K)))$ whose restriction to $\pi_1(X_2)$ (and hence also $\pi_1(Y)$) are abelian and representations of $\pi_1(\surg{1}{2}{K})$.  

Every $SU(2)$ representation of $\surg{1}{2}{K} $ restricts to an abelian representation on the solid torus $S^1\times D^2$. Moreover, since $\surg{1}{2}{K} $ is a homology sphere, every non-trivial $SU(2)$ representation of $\pi_1(\surg{1}{2}{K}) $ necessarily restricts to a non-abelian representation of $\pi_1(X_1)$.
  Thus the above correspondence induces a well-defined bijection on conjugacy classes
$$ \chi_{SU(2)}(\surg{1}{2}{K} )\setminus \{\theta\}\ \ \longleftrightarrow \ \ \chi_{SU(2)}(\branch (\doub(K)))_\sss{NAA}.$$

Precisely, the correspondence is given by extending a representation over $N$, or, in the context of connections, extending   a flat connection over the cobordism $N$.  Since the Chern-Simons invariant is a flat cobordism invariant modulo $\Z$,  the correspondence preserves Chern-Simons invariants modulo $\Z$, as asserted. 
 \end{proof}

 \subsection{Whitehead doubles of torus knots}
 The Chern-Simons invariants for the 2-fold branched cover Whitehead doubles of torus knots can now be estimated.  

\begin{theorem}\label{thm3.9} Let $p,q>0$ be relatively prime, and let $T_{p,q}$ be the $(p,q)$ torus knot.\hfill

\begin{enumerate} \item  Suppose $\gamma \in \chi_{SU(2)}(\branch (\doub(T_{p,q})))_\sss{NAA}$. Then $\cs(\gamma)$ is a rational number with denominator dividing $4pq(2pq-1)$.
\item Suppose $\gamma \in \chi_{SU(2)}(\branch (\doub(T_{p,q})))_\sss{NAN}$. Then $\cs(\gamma)$ is a rational number with denominator dividing $4pq(4pq-1)$
\end{enumerate}
These statements remain true with the orientation of $\branch (\doub(T_{p,q}))$ reversed. 
Hence $$\tau(\pm \branch (\doub(T_{p,q})),\theta)\ge \frac{1}{pq(4pq-1)}.$$
\end{theorem}

\begin{proof}    Proposition \ref{lem6.3} implies that  the values of  Chern-Simons invariants modulo $\Z$ of representations $\tilde \gamma \in \chi_{SU(2)}(\branch (\doub(K)))_{NAA}$ coincide with the 
the values of Chern-Simons invariants of non-trivial representations $\tilde \gamma \in  \chi_{SU(2)}(\surg{1}{2}{K} )$.  When $K$ is a $(p,q)$-torus knot, then $\surg{1}{2}{K} = -\Sigma(p,q,2pq-1)$.  The values of Chern-Simons
invariants of flat $SU(2)$  connections on $\Sigma(p,q,2pq-1)$ are rational with denominator dividing $4pq(2pq-1)$ (\cite{FS},  \cite[Theorem 5.2]{KK1}, \cite[Lemma 3.2]{HK}).

This leaves the $NAN$ representations. The Chern-Simons invariants can be computed explicitly  using the torus decomposition  approach of \cite{KK1,KK2}.  We give an alternative and simpler argument, which is sufficient to establish the estimate.  This argument is inspired by Fintushel-Stern's method of computing Chern-Simons invariants of flat $SO(3)$ connections on Seifert-fibered homology spheres, and informed by Klassen's identification of the representation spaces of  torus knots \cite{K}.

Consider $\tilde \gamma\in  \chi_{SU(2))}(\branch (\doub(K)))_\sss{NAN}$. The restriction  of $\tilde \gamma$ to  $X_1$,  $\tilde \gamma:\pi_1(X_1)\to SU(2), $ is non-abelian.  Since   $X_1$ is the $(p,q)$ torus knot complement, it is Seifert fibered over $D^2$. The regular fiber $F_1\subset X_1$ represents a central element in $\pi_1(X_1)$, and since the centralizer of any non-abelian subgroup of $SU(2)$ is just $\pm 1$, $F_1$ is necessarily sent to $\pm 1$ by $\tilde \gamma$.  Note that  $F_1$ can be taken to lie  on $T_1$ and is represented by the loop $\mu_{K_1}^{pq}\lambda_{K_1} $. Hence $\tilde \gamma(\mu_{K_1}^{pq}\lambda_{K_1})=\pm 1$.  The same argument applied to $X_2$ shows that   $\tilde \gamma(\mu_{K_2}^{pq}\lambda_{K_2})=\pm 1$.
Therefore, the composite $${\gamma}:\pi_1(\branch (\doub(K)))\xrightarrow{\tilde \gamma}SU(2)\to SO(3)$$ sends the loops
$\mu_{K_1}^{pq}\lambda_{K_1}$ and $\mu_{K_2}^{pq}\lambda_{K_2}$ to $1$.

The mapping cylinder $M_1$ of the Seifert fibration $X_1\to D^2$ is an orbifold with two singular points: a cone on $L(p,q)$ and a cone on $L(q,p)$. The boundary of $M_1$ is the Dehn filling of $X_1$ which kills the regular fiber  $\mu_{K_1}^{pq}\lambda_{K_1}$, i.e. $ {pq}$-surgery on the $(p,q)$ torus knot. 
Let $M_1^{\circ}$ denote the complement of the neighborhoods of the cone points.   A straightforward calculation shows that the inclusion $\pi_1(X_1)\to \pi_1(M_1^{\circ})$  is surjective, with kernel generated by $F_1$.   Thus the restriction of ${\gamma}$ to $\pi_1(X_1)$ extends to $\pi_1(M_1^{\circ})$. Again, the mirror argument provides a manifold $M_2^{\circ}$ with boundary the union of $L(p,q), L(q,p)$ and the Dehn filling of $X_2$ which kills $\mu_{K_2}^{pq}\lambda_{K_2}$, over which ${\gamma}$ extends. 

Glue $M_1^{\circ}$, $[0,1]\times Y $, and $M_2^{\circ}$ together along neighborhoods $[0,1]\times T_1 $ and $[0,1]\times T_2 $ of their boundaries.  This produces a 4-manifold
$$M=M^{\circ}_1\underset{[0,1]\times T_1 }\cup  ([0,1]\times  Y ) \underset {[0,1]\times T_2 } \cup M^{\circ}_2$$
with boundary 
$$\partial M=-\branch (\doub(K))\sqcup 2L(p,q)\sqcup 2L(q,p)\sqcup  \hat{Y}$$
over which ${\gamma}$ extends.    We denote this extension again by $\gamma$:
$$\gamma:\pi_1(M)\to SO(3).$$
Here $\hat{Y}$ denotes the Dehn filling of the two components of the boundary of $Y$ in which $\mu_{K_1}^{pq}\lambda_{K_1}= \mu_{A_1}^{1-2pq}\lambda_{A_1}^{pq}$
(by Proposition \ref{2foldcover})  and $\mu_{K_2}^{pq}\lambda_{K_2}= \mu_{A_2}^{1-2pq}\lambda_{A_2}^{pq}$ are killed. In  other words, $\hat{Y}$ is obtained from the $(2,4)$ torus link by performing $\frac{1-2pq}{pq}$-surgery on each component.  

Notice that the restriction of the representation ${\gamma}$ to $\hat{Y}$ is abelian. Moreover, $H_1(\hat{Y})$ is presented by the $2\times 2$ matrix 
$$\begin{pmatrix} 1-2pq&2pq\\2pq&1-2pq
\end{pmatrix}.$$
Thus $H_1(\hat Y )=\Z/(4pq-1)$.

Chern-Simons invariants modulo $\Z$ are flat cobordism invariants, and are additive over disjoint unions, and so
 $\cs(\gamma, \branch (\doub(K)))=-4 \cs(\tilde\gamma, \branch (\doub(K)))$ is the sum of five $SO(3)$ Chern-Simons invariants: the four Chern-Simons invariants of the restriction of ${\gamma}$ to the lens spaces, and the Chern-Simons invariant of the corresponding abelian representation  of $\hat Y$.  
 
 The Chern-Simons invariant of any abelian $SO(3)$ representation with finite image on any closed 3-manifold   is a rational number with denominator the order of the image.  This can be seen by integrating the Chern-Simons form over a fundamental domain in the corresponding finite cover.  Hence  $\cs(\gamma, \branch (\doub(K)))$ is the sum of five fractions, with denominators $p,q,p,q,4pq-1$. Thus it  is a fraction with denominator dividing $pq(4pq-1)$.
 
Reversing the orientation of a 3-manifold changes the signs of  its Chern-Simons invariants.  Renormalizing to place its value in $(0,4]$ by adding $4k$ does not alter the denominator. It follows that these estimates are valid with either orientation.   The last statement is a consequence of   Corollary \ref{cor6.2}. 
\end{proof}

We can now prove our main result, which includes   Theorem 1 of the introduction.
 
 \begin{theorem}\label{mainresult}
Let $\{(p_i,q_i)\}_{i=1}^\infty$ be a sequence of relatively prime integers satisfying 
$$ {p_nq_n(2p_nq_n-1)}>{p_{n-1}q_{n-1}(4p_{n-1}q_{n-1}-1)},$$ 
 For example,  $(p_n,q_n)=(2, 2^n-1)$, or  $(p_n,q_n)=(n, nk_n-1)$ where $k_n>\sqrt{2} k_{n-1}$.

 Then the positive-clasped untwisted Whitehead doubles $\{\doub(K_{p_i,q_i})\}_{i=1}^\infty$ of the  $(p_i,q_i)$ torus knots  are independent in the smooth concordance group. 
\end{theorem} 

\begin{proof}  
Corollary \ref{prop5.1}  establishes the existence of 4-manifolds $W_i=W_{p_i,q_i}$ satisfying Property \inst \ with respect to classes $e_i\in H^2(W_i)$, with $-e_i^2= \tfrac{1}{p_iq_i(2p_iq_i-1)}$. Moreover, the boundary of $W_i$ admits a partition 
$$\partial_\gl W_i= -\Sigma(D(T_{p_i,q_i}), \ \ \partial_\gl W_i=Y_{1,i}\sqcup Y_{2,i}\sqcup Y_ {3,i}$$
which is a $\cs$-partition for {\em both} $(W_i, e_i, -e_i^2)$ and $(W_i, 0, -e_i^2)$.

 Note that 
 $$  p_nq_n(2p_nq_n-1) > p_{n-1}q_{n-1}(4p_{n-1}q_{n-1}-1) >p_{n-1}q_{n-1}(2p_{n-1}q_{n-1}-1)$$
 and hence  if $1\leq j<i$, 
 $$  p_iq_i(2p_iq_i-1) \ge p_{j+1}q_{j+1}(2p_{j+1}q_{j+1}-1)>p_{j}q_{j}(4p_{j}q_{j}-1) .$$
  Theorem \ref{thm3.9}  gives the estimate $\tau(\pm\Sigma(T_{p_j,q_j}))\ge \tfrac{1}{p_jq_j(4p_jq_j-1)}$, and therefore
$$\tau(\pm\Sigma(T_{p_j,q_j}))\ge \tfrac{1}{p_jq_i(4p_jq_j-1)}>\tfrac{1}{p_{i}q_{i}(2p_{i}q_{i}-1)}=-e_i^2.
 $$  
The result now follows from Corollary \ref{cor:indknot} \end{proof}

 \section{Comparison with other techniques}\label{compare}
 Many striking developments have occurred in the twenty-five years since instantons were first used to study cobordism theory.   In this section we view our results in the light of these advances.

We begin by pointing out that while significant progress has been made in understanding the topological concordance group,  these techniques cannot be brought to bear on the questions addressed here.  As mentioned in the introduction, the Whitehead double of any knot is topologically slice. Hence any invariant which can prove independence of Whitehead doubles in the concordance group must be manifestly smooth. 

In this direction, the past decade has seen an explosion of activity in the study of smooth concordance.  The activity has been facilitated by a wide variety of newly discovered invariants.  Broadly speaking, these invariants come in two flavors;  invariants of knots and three-manifolds 
defined in the context of Heegaard Floer homology (or several other conjecturally equivalent theories, e.g. monopole Floer homology), and invariants of knots derived from Khovanov homology and its generalizations.  The former invariants are analytically defined, while the latter are defined algebraically in the context of quantum algebra and its categorification.    Since both approaches have led to striking new applications, the reader may be surprised that instantons were the tool of choice in the problem at hand.  

A distinguishing feature of our approach is that, unlike many of the modern invariants, the instanton moduli spaces used here {\em do not} provide homomorphisms from the  concordance group to  the integers.   When using homomorphisms to $\Z$ to detect  independence in $\mathcal{C}$, one needs at least as many homomorphisms as there are knots in the set whose independence one would like to establish.  Since we consider infinite sets of Whitehead doubles, a proof of our theorem in such contexts would necessarily require the computation of an infinite number of invariants for each knot.  In contrast, our proof required essentially two computations for each knot in the family: first, the verification that  the branched double cover bounds a 4-manifold with Property \inst  \ and second, the  calculation of a lower bound for the minimal Chern-Simons invariant.   

Despite this, one may still hope that the modern invariants could prove our theorem, a possibility which we now discuss.   Let us first examine the natural approaches using Khovanov homology. In this realm one has invariants $s_n(K)$, $n=2,... \infty$  which are related to smooth 4-dimensional topology by the inequality \begin{equation}\label{eq:s} |s_n(K)|\le 2(n-1)g_4(K),\end{equation} where $g_4(K)$ is the smooth 4-ball genus of $K$, \cite{RasSlice,Wu2009,Lobb}.  Here, $s_2(K)$ is  Rasmussen's invariant \cite{RasSlice}, defined via Lee's \cite{Lee} perturbation of Khovanov homology \cite{Khovanov2000}; the invariants for $n>2$ are defined by similar techniques applied in the context of Khovanov and Rozansky's link homology theory for $sl_n$ \cite{KhovRoz}. The remarkable feature of these invariants is that the inequality is sharp for torus knots, thus proving Milnor's conjecture through purely combinatorial means.  It is expected that, like $s_2$, each invariant provides a homomorphism $s_n:\mathcal{C}\rightarrow \Z$, and one might hope that this collection of homomorphisms could be used to prove the independence of Whitehead doubles of torus knots.  Unfortunately, this approach cannot work.  The key point is that Whitehead doubles of torus knots are {\em quasipositive}.  This means they can be presented as the closure of a braid which is a product of conjugates of positive half-twists or, more geometrically, that they arise as the intersection of a smooth algebraic curve in $\C^2$ with the unit-sphere.    Equation \eqref{eq:s} is sharp  for quasipositive knots \cite{Olga2006,Shumakovitch,Wu2008}. Since the Whitehead doubles of torus knots have 4-genus equal to $1$, the collection of invariants $\{ s_n(K)\}$ would therefore be too weak to prove their independence.  

A second approach is to make use of the observation  that satellite operations induce functions $\mathcal{C}\rightarrow \mathcal{C}$.  Precomposing $s_n:\mathcal{C}\rightarrow \Z$ with such a function seems likely to contain very interesting information.  However, examining these invariants for  satellites of  connected sums of Whitehead doubles is far beyond the present computational reach.  
 
Having dispensed with the known invariants from quantum algebra, we turn to invariants defined by Floer homology.    The most well-known of these is the \os \ concordance invariant, $\tau(K)$.  This invariant shares several key features of $s_2(K)$ but is, in general, distinct from it \cite{Stau}.  By itself, however, $\tau$ can only show that Whitehead doubles of torus knots generate a free abelian subgroup of rank one, as it is a homomorphism from $\mathcal{C}$ to $\Z$.  In fact $\tau(D(T_{p,q}))=1$ for all $p,q,$ since $\tau$ satisfies a genus inequality analogous to  \eqref{eq:s}, which is also sharp for quasipositive knots \cite{SQPfiber,Livingston}.  Thus we must either find an infinite collection of additional homomorphisms or use a different approach.   

Both avenues may be pursued through the use of the ``correction terms" in \os  \ Floer homology (see \cite{AbsGrad} for a definition).  These are defined in terms of the grading and module structure on the Floer homology groups, and provide invariants  of Spin$^c$ rational homology cobordism.   As we have seen, by applying topological constructions to knots, we can obtain maps from the concordance group to various cobordism groups.  Two natural maps are provided by forming branched covers and performing Dehn surgery.  We have already seen the first map through Lemma \ref{branchslice}.  Indeed, by considering  prime power order cyclic branched covers, we obtain an infinite family of homomorphisms  $$\Sigma_{p^k}:  \mathcal{C} \rightarrow \Theta_{\Z/p},$$ where $\Sigma_{p^k}$ denotes the map which takes the concordance class of a knot to the $\Z/p$-homology cobordism class of its $p^k$-fold branched cyclic cover (here, $p$ is a prime).   One can obtain homomorphisms, $\delta_{p^k}:\mathcal{C}\rightarrow \Theta_{\Z/p} \rightarrow \Z$ by considering the correction term of a particular Spin$^c$ structure on these covers \cite{MO,Jabuka}.  Although we have not calculated these invariants (except in the case $\delta_2$, where they were calculated in \cite{MO}), we have used cobordism arguments similar to those presented here to obtain bounds which tightly constrain the behavior of $\delta_{p^k}(D(T_{r,s}))$.  Indeed, it seems likely that the essential information contained in these invariants is that of the correction term for $+1$ surgery on $T_{r,s}$.  Heuristically, the fact that the  branched cyclic covers of $D(T_{r,s})$ decompose along tori into  a standard manifold  glued to a number of copies of $S^3\setminus n(T_{p,q})$ should destroy the chance of independence of the $\delta$ invariants applied to Whitehead doubles.  This could be justified rigorously if the $(2+1)$ dimensional TQFT inherent in Floer homology were more computable.  Similar remarks should apply to the concordance invariants derived by considering $\tau$ of satellites.

On the other hand, one could break free from homomorphisms by applying the correction terms to Dehn surgery on linear combinations of Whitehead doubles of torus knots. Like the branched covering construction, Dehn surgery provides maps from the concordance group to cobordism groups.  These maps are more like those induced by satellite constructions, and are not homomorphisms.  Unfortunately, this approach fails, due to the fact that the knot Floer homology invariants of Whitehead doubles are well-understood \cite{Doubling}, together with the fact that the correction terms of Dehn surgery on a knot are determined by the knot invariants \cite{Knots}.  For instance, if one considers $D(T_{2,3})-D(T_{2,7})$, the part of the knot Floer homology complex of $D(T_{2,3})\#-D(T_{2,7})$ relevant to the correction term calculations is identical to the knot Floer homology complex of $T_{2,3}\#-T_{2,3}$, a slice knot (see \cite{HLR} for details on computing correction terms for surgeries on Whitehead doubles).  Thus any concordance information extracted from the correction terms of Dehn surgeries on $D(T_{2,3})\#-D(T_{2,7})$ would necessarily be trivial.   

Thus it appears that moduli spaces of instantons contain information obscured, or perhaps even lost, by the modern invariants.   The key fact seems to be the appearance of $SO(3)$ Chern-Simons invariants as the units of energy for instantons in the ends of the moduli spaces.   These invariants carry information about the fundamental group which is presently missing in modern Floer theories or Khovanov homology.  This also underlies the lack of an instanton-free proof of Property $P$.  For these reasons, we find further study of non-abelian gauge theory on 3- and 4-manifolds a well-motivated pursuit.  Indeed in the realm of concordance alone, the instanton obstruction provides a powerful  and complementary technique to those afforded by the Khovanov or \os \ theories.   In light of this, we feel the technique would be a valuable addition to the toolbox of anyone interested in the smooth concordance group.

\end{document}